\documentclass{amsart}
\usepackage{latexsym,amssymb,amsmath,enumitem}
\usepackage{verbatim}
\usepackage[usenames,dvipsnames]{xcolor}
\usepackage{cases}

\newcommand{\mfkA}{\mathfrak{A}}

\newcommand{\RR}{{\mathbb R}}
\newcommand{\NN}{{\mathbb N}}
\newcommand{\ZZ}{{\mathbb Z}}

\newcommand{\CC}{{\mathbb C}}

\newcommand{\LL}{{\mathbb L}}

\newcommand\CIb{\mathcal C_b^\infty}

\newcommand{\cB}{{\mathcal B}}
\newcommand{\cC}{{\mathcal C}}
\newcommand{\cD}{{\mathcal D}}

\newcommand{\cF}{{\mathcal F}}
\newcommand{\cG}{{\mathcal G}}

\newcommand{\cL}{{\mathcal L}}

\newcommand{\cV}{{\mathcal V}}

\newcommand{\D}{{\partial }}

\newcommand{\bR}{\mathbb{R}}

\newcommand{\bL}{\mathbb{L}}

\newcommand{\bC}{\mathbb{C}}

\newcommand\ie{{\em i.e., }}

\newcommand\adj{\operatorname{ad}}

\newtheorem{theorem}{Theorem}[section]
\newtheorem{proposition}[theorem]{Proposition}
\newtheorem{corollary}[theorem]{Corollary}
\newtheorem{lemma}[theorem]{Lemma}
\theoremstyle{definition}
\newtheorem{definition}[theorem]{Definition}

\theoremstyle{remark}
\newtheorem{remark}[theorem]{Remark}

\numberwithin{equation}{section}

\newcommand{\beq}{\begin{equation}}
\newcommand{\eeq}{\end{equation}}

\newcommand{\<}{\langle}
\renewcommand{\>}{\rangle}
\renewcommand{\a}{\alpha}

\newcommand{\s}{\sigma}

\newcommand\ede{\, := \,}
\newcommand\seq{\, = \,}

\usepackage[normalem]{ulem}
\renewcommand\sout{\bgroup\markoverwith
{\textcolor{red}{\rule[0.7ex]{3pt}{1.4pt}}}\ULon}

\newcommand\jap[1]{\langle #1 \rangle}
\renewcommand\adj{\operatorname{ad}}

\definecolor{DarkRed}{rgb}{0.8,0.3,0.6}

\definecolor{Green}{rgb}{0.010,0.7,0.02}

\newcommand{\pa}{\partial}
\newcommand{\fA}{\mathfrak{A}}

\makeatletter

\newcommand{\Rmnum}[1]{\expandafter\@slowromancap\romannumeral #1@}
\makeatother

\title[Parabolic Equations]{Approximate Solutions to Second-Order
  Parabolic Equations: evolution systems and discretization}

\author[W. Cheng]{Wen Cheng}
\email{cheng@math.psu.edu}
\address{Equity Derivative Quantitative Strategies, Credit Suisse, 
11 Madison Avenue, New York, NY 10010, USA}

\author[A. Mazzucato]{Anna L. Mazzucato}
\email{alm24@psu.edu}
\address{Penn State University, Mathematics Department, University Park, PA 16802,
  USA}
  
\author[V. Nistor]{Victor Nistor}
\email{victor.nistor@univ-lorraine.fr}
\address{Universit\'e de Lorraine, 3, rue Augustin Fresnel, 57000 Metz, France}

\thanks{A.M. was partially supported by the US National Science Foundation grant
DMS-1909103.
 }

\date\today

\begin{document}

\begin{abstract}
We study the discretization of a linear evolution 
partial differential equation when its Green function is known. 
We provide error estimates both for the spatial approximation 
and for the time stepping approximation. We show that, in fact,
an approximation of the Green function is almost as good
as the Green function itself. For suitable time-dependent
parabolic equations, we explain how to obtain good, explicit approximations
of the Green function using the Dyson-Taylor commutator method (DTCM) 
that we developed in {\em J. Math. Phys.} (2010). This approximation for 
short time, when combined with a bootstrap argument, gives 
an approximate solution on any fixed time interval within 
any prescribed tolerance.
\end{abstract}

\maketitle

\begin{center}
\dedicatory{\em In loving memory of Rosa Maria (Rosella) Mininni}
\end{center}

\tableofcontents

\section{Introduction}
\label{sec1}

We consider an initial value problem (IVP) of the form
\begin{equation}\label{eq.def.IVP}
\begin{cases}
    \ \pa_t u (t) - L(t) u(t) \seq f  \\
    \ u(s) \seq h \,.
\end{cases}
\end{equation}
We require $u(t)$ and $h$ to belong to suitable given Sobolev spaces
on some $\Omega \subset \bR^N$. If $\Omega \neq \bR^N$, we impose suitable 
boundary conditions. Let us assume $f = 0$. The {\em solution operator,} 
if it exists, is then $U^L(t, s)h = u(t)$;
it defines what is called an {\em evolution system} \cite{Amann, Lunardi, Pazy}
(we also recall the definition of an evolution system below, see Definition 
\ref{def.evolutionsys}). We have
\begin{equation}\label{eq.def.Green}
    \big[ U^L(t, s)h \big] (x) \seq \int_{\Omega} \cG_{t, s}^L(x, y) h(y) dy\,,
\end{equation}
when such a distribution $\cG_{t, s}^L(x, y)$ exists. We call this distribution
the {\em Green function} $\cG_{t, s}^L(x, y)$ of the evolution system $U^L$.
(In the cases considered in this paper, it will be a {\em true function.} We
shall also say that $\cG_{t, s}^L(x, y)$ is the Green function of $\pa_t - L$.)


In this paper we address the following questions:

\begin{enumerate}
\item Assuming that the Green function $\cG_{t, s}^L(x, y)$ of
the evolution system $U^L$ is known, to establish the properties of 
the approximations of $u(t)$ in suitable discretization spaces $S$.

\item To show that suitable good approximations of the Green function
are (almost) as good as the Green function itself. 

\item To provide a method to find good approximations of the 
Green function, including complete error estimates.
\end{enumerate}

To obtain significant results on the above questions, we shall make
suitable assumptions. First, we assume that
\begin{equation}\label{eq.operator}
  L \ede \sum_{i,j}^{N} a_{ij}(t,x) \partial_i \partial_j + \sum_{i}^{N}
  b_i(t,x) \partial_i + c(t,x),
\end{equation}
with $x = (x_1, ..., x_N) \in \bR^N$ and $\D_k := \frac{\D}{\D x_k}$.
Almost nothing changes in most of our discussion if we consider nice 
subsets $\Omega \subset \bR^N$
instead $\bR^N$, except that we have to consider boundary conditions. 
So we shall work on $\bR^N$ in this paper.
The coefficients $a_{ij}$, $b_i$, and $c$ are assumed smooth and bounded and all
their derivatives are assumed to be uniformly bounded (\ie, they are
assumed to be in $W^{\infty, \infty}(\bR^+\times \bR^N) = \cC_b^\infty(\bR^+\times\bR^N)$.)
For simplicity, we assume that  $a_{ij}=a_{ji}$ as well. We impose a uniform strong ellipticity 
condition on the operators $L(t)$, meaning that there exists a constant $\gamma>0$, 
such that
\begin{equation}\label{eq.uniformly.elliptic}
  \sum a_{ij}(t,x)\xi_i \xi_j \geq \gamma \|\xi\|^2, \quad \forall
   t\ge 0, \ x, \,\xi \in \bR^N,\; \xi\ne 0.
\end{equation}
We collectively denote by $\LL_\gamma$ the class of operators $L$ of the
form \eqref{eq.operator} satisfying the ellipticity condition
\eqref{eq.uniformly.elliptic} and  the coefficients of which, together
with all their derivatives, are bounded (see Definition \ref{def.Lgamma}).

Let us address in more detail the three main contributions of this paper, 
listed above.
\smallskip

(1)\ 
The first contribution  ({\em``Assuming that the Green function $\cG_{t, s}^L(x, y)$ 
of the evolution system $U^L$ is known, to establish the properties of the 
approximations of $u(t)$ in suitable discretization spaces $S$''}) addresses a 
very natural question. Even if, theoretically, the knowledge of the initial data 
$h$ and of the Green functions $\cG_{t, s}^L(x, y)$ determines the solution
$u$ via integration: $u(t, x) = \int_{\bR^N} \cG_{t, 0}^L(x, y) h(y) dy$,
in practice, two other issues arise. The first one is that we can store in
the memory of our computer only a {\em finite dimensional} space $V$ of solutions.
We thus need to {\em discretize} our equation and to {\em approximate} both
the initial data and the solution with elements of $V$. Our first result,
Theorem \ref{thm.disc.error} gives a ``proof of concept'' result on how such
a discretization (in the space variable) works. The main point of the
result is that the projection error has to decrease in time at the same
order as the time itself (see Theorem \ref{thm.disc.error}, especially
the Condition \ref{eq.def.cond.com}). In our setting, we know few error estimates of 
this kind, but in the general framework of Finite Difference or Finite Element
methods for evolution equations, there are some similar results
\cite{GlowinskiLions, schwabFinBook, Thomee, Samarskii, LarssonThomee}.

(2)\
Our second contribution ({\em ``To show that suitable good approximations of the 
Green function are (almost) as good as the Green function itself''})
addresses another natural question, which is what kind of approximations
of the Green functions would be acceptable, in case the Green function
itself is not known? Assume that an approximate Green function 
$\widetilde{\cG}_{t, s}^L(x, y)$ is given. Assume also that the
discretization in space is to divide the time interval $[0, T]$ in
$n$ equal size intervals (we will always use this very common procedure).
If the error $\|\cG_{t, s}^L - \widetilde{\cG}_{t, s}^L\|$ is of
the order of $(t - s)^\alpha$, then we show that the order of the error
due to time discretization (or bootstrap) is of the order $n^{1-\alpha}$.
This shows that we need a rather good approximation of the Green
function ($\alpha > 1$). The bootstrap method is the one we developed in 
\cite{CCLMN, CMN-preprint}. It is a common method in Finite Difference and 
Finite Element methods \cite{GlowinskiLions, schwabFinBook, Thomee, Samarskii, 
LarssonThomee}. For Green functions, a similar method was more recently suggested 
in \cite{Pascucci15}.

A common issue in both space and time discretization (\ie, in (1) and (2))
is that we need to find error estimates that are at least of the order
of $(t-s)$ (in fact, even better for (2)). We know very few earlier results 
in the line of (1) and (2).
\smallskip

(3)\
Our third contribution {\em (``To provide a method to find good approximations of 
the Green function, including complete error estimates'') } fits into a very long
sequence of results concerning heat kernel approximations and Dyson series
expansions. Here the list of papers that one could quote is truly Gargantuan,
but let nevertheless mention the papers \cite{Berestycki02, Berestycki04,  
FJacquier09, FPSS03, GatheralJacquier11, Les03, Labordere05, Labordere07} among 
earlier papers most closely related to our work \cite{CCMN, CCLMN, WenThesis, 
CMN-preprint} (in chronological order), in which we have developed the 
Dyson-Taylor commutator method used in this paper. 
Let us mention also the more recent papers \cite{Gatheral12a, 
Gatheral12b, Xiao, GHJ18, Les15, Les17, Pascucci17, Siyan}, where the reader will be
able to find further references. Some general related monographs 
include \cite{FPSS11, LabordereBook, LewisBook00}.

For the Green function approximation, we use the Dyson-Taylor commutator method (DTCM) 
developed in \cite{CCMN, CCLMN, WenThesis, CMN-preprint}, which we also expand and
make more precise. A similar method was employed more
recently in \cite{Pascucci15, Pascucci17}. The main result regarding this third questions 
is a sharp error estimate in weighted Sobolev spaces.
This error estimate, when combined with the results of (2) and
using the bootstrap argument we developed \cite{CCLMN} gives 
an approximate solution on any fixed time interval within 
any prescribed tolerance. Our method is such that also derivatives of the
solution can be effectively approximated with verified bounds (with the price of 
increasing the order of approximation). Our error estimates are in 
{\em exponentially weighted Sobolev spaces} $W_a^{r, p}(\bR^N) 
= e^{-a \jap{x} } W^{r, p}(\bR^N)$.

Our main result is the following. (The class $\bL_\gamma$ was introduced
above, but see also \ref{def.Lgamma}.)

\begin{theorem}\label{thm.maintheorem}
Let $L$ be an operator in the class
$\LL_\gamma$. Then $L$ generates an evolution system $U^L$ in the
Sobolev space $W^{r,p}_{a}(\RR^N)$, $r\geq 0$, $1<p<\infty$, $a\in
\RR$. Given $m \in \NN$, there exists an explicitly computable
smooth function $\cG_{t, s}^{[m,z]}(x,y)$, given in Definition \ref{def.orderm},
such that the Green's function for $\pa_t-L$ can be represented as  
\begin{equation*}
  \cG_{t, s}^{L}(x,y) := \cG_{t,s}^{[m]}(x,y) 
  + (t-s)^{(m+1)/2} \widetilde{E}^{t,s}_{m}(x,y)\,,
\end{equation*}
where the remainder $\widetilde{E}^{t,s}_{m}$, 
when regarded as an integral operator, satisfies
\begin{equation*}
  \|\widetilde{E}^{[m]}_{t, s} g\|_{W^{r+k,p}_{a}} \leq C\,
  (t-s)^{-k/2}\|g\|_{W^{r,p}_{a}}, \quad 0 \le s < t \le T,\ k \in \NN
\end{equation*}
with a bound $C$ depending on $L, m, a, k, r, p,z$, and
$0<T<\infty$, but  independent of $g$ and $s, t\in [0,T]$, $s \le t$.
\end{theorem}

Together with Theorem \ref{thm.b.error}, this yield an approximation
of the solution $u$ of our Initial Value problem \eqref{eq.def.IVP}.

The paper is
organized as follows. In Section \ref{sec2}, we remind some standard
facts about non-autonomous, second-order initial value problems $(\partial_t-L(t))u(t,x)=0$
and the evolution system they generate. In Section \ref{sec3}, we establish
space discretization and time discretization (bootstrap) error estimates in 
a general, abstract setting. The setting is that of an evolution system 
that satisfies some standard exponential bounds. These exponential bounds
are satisfied both in the parabolic and hyperbolic settings, so they are
realistic. Beginning with Section \ref{sec4}, we specialize to the case of operators 
$L \in \bL_{\gamma}$. In that section, we introduce weighted Sobolev spaces,
we study the evolution system generated by $L \in \bL$. Using the theory of 
analytic semigroups, we establish explicit
mapping properties that allow us to make sense of the integrals appearing in
the iterative time-ordered expansions that we use (the resulting formulas
are sometimes called Dyson-series and are well known and much used in the 
Physics literature). The time-ordered expansion is obtained, as usual,
using Duhamel's principle iteratively. Section \ref{sec5} contains a formal 
derivation of the asymptotic expansion of the solution operator for the Equation
\eqref{eq.def.IVP}. This derivation allows us to use the  method from \cite{CCMN} 
for computing the time-ordered integral appearing in the resulting Dyson series 
expansion using Hadamard's formula:
\begin{equation} \label{lemma.hadamard}
e^A B = \left( B+  [A,B] + \frac{1}{2!}[A, [A,B]] + \frac{1}{3!}[A,
[A,[A,B] ]  ] + \dots \right) e^{A}.
\end{equation}
Here we use the crucial observation in \cite{CCMN} that, in the cases of interest 
for us, this series reduces to a finite, explicit sum. 
In Section \ref{sec6}, we introduce our approximate Green function, 
we prove Theorem \ref{thm.maintheorem}, and 
we complete our error analysis. Technically, this section is one
of the most demanding.

Throughout the paper, unless explicitly mentioned, $C$ will denote a
generic constant that may be different each time when it is used.
We employ standard notation for function spaces throughout, in particular 
$W^{r,p}$, $1\leq p\leq \infty$, $r\in \RR$  for standard $L^p$-based Sobolev 
spaces on $\RR^n$, and $H^s=W^{s,2}$. We also denote the space
of continuous functions (which may take values in a Banach space) with $\cC$, and by 
$\cC^\infty_b$ the space of smooth functions with bounded derivatives of all orders.

The results of this paper are based in great part and extend some results 
in \cite{WenThesis} and an unpublished 2011 IMA preprint \cite{CMN-preprint}. 
See \cite{Xiao, Pascucci17, Siyan} for some recent, related results to
that preprint. However, Section \ref{sec3} is essentially new. Also, we did
not include the numerical test and the explicit calculations of the SABR model
from \cite{CMN-preprint}, to keep this paper more focused (and to limit its size).

\medskip

\noindent {\bf Convention:} {\em we use throughout the usual
  multi-index notation for derivatives with respect to the {\em space}
  variable $x$, that is, \ $\pa^\alpha = \pa^{\alpha_1}_{1} \ldots
  \pa^{\alpha_N}_{N}$, \ $\alpha=(\alpha_1,\ldots,\alpha_N) \in \ZZ_+^N$, and
  \ $|\alpha|=\sum_{i=1}^N \alpha_j$,
$\pa_j = \frac{\pa}{\pa x_j}$, while  $\pa_t = \frac{\pa}{\pa t}$.
 Throughout, we fix an arbitrary $0<T<\infty$.}

\medskip

\subsection*{Acknowledgments}
The authors would like to thank Radu Constantinescu, Nicola
Costanzino, John Liechty, Jim Gatheral, Christoph Schwab, and
Ludmil Zikatanov for useful discussions. 

\section{Preliminaries on evolution systems \label{preliminary}}
\label{sec2}

We refer the reader to \cite{Amann, Lunardi, Pazy} for the functional 
analytic framework we employ.

Let $X$ be a Banach space and let $A : \cD(A) \to X$ be a 
(possibly unbounded) closed linear operator with domain
$\cD(A) \subset X$. We let $\rho(A)$ denote its resolvent set, that is,
the set of $\lambda \in \CC$ such that $A - \lambda : \cD(A) \to X$ is
a bijection. We let $R(\lambda, A) := (\lambda - A)^{-1} : X \to X$
be its {\em resolvent}, for $\lambda \in \rho(A)$.

\begin{definition} \label{def.sectorial}
A closed operator $A : \cD(A) \to X$ is called {\em sectorial} if there
are constants $\omega \in \RR$, $\theta \in (\pi/2,\pi)$, and $M>0$
such that
\begin{equation*}
\left \{
\begin{array}{l}
  \rho(A) \supset S_{\theta,\omega} := \{\lambda \in \CC ,\,
  \lambda \neq \omega, |arg(\lambda-\omega)|<\theta\},\\[1mm]
  \|R(\lambda,A)\| \leq M/|\lambda-\omega|, \ \forall
  \lambda \in S_{\theta,\omega} \,.
\end{array}
\right.
\end{equation*}
\end{definition}

As discussed later, sectoriality implies mapping properties for the
evolution system $U(t,r)$, generated by $L(t)$. The following well known 
proposition (see again \cite[page 43]{Lunardi} for
a proof) gives a sufficient condition that  guarantees the sectoriality of an
operator.

\begin{proposition}\label{condition.on.sectorial}
Let $A$ be a linear operator $\cD(A)\subset X \to X$.
Assume that there exists $\omega \in \RR$ and $M>0$  such that
$\rho(A)$ contains the half plane $\{\lambda \in \bC, \, \operatorname{Re}
  \lambda \geq \omega\}$ and
\begin{equation*}
  \|\lambda R(\lambda,A)\|_{\cL(X)} \leq M, \quad \forall
  \operatorname{Re} \lambda \geq \omega.
\end{equation*}
 Then $A$ is sectorial.
\end{proposition}

\subsection{Properties of evolution systems}
\label{evolutionsystem}

In this section, we show that $L(t)$ generates an evolution system on Sobolev 
spaces. We recall below the definition of an evolution system and some basic 
properties for the reader's convenience. (We refer to \cite{Lunardi} for an 
in-depth discussion. See also \cite{Amann, Pazy})


\begin{definition} \label{def.evolutionsys}
Let $I \subset [0, \infty)$ be an interval containing 0.
A two parameter family of bounded linear operators $U(t,t')$ on $X$,
$0\leq t' \leq t$, $t', t \in I$, is called {\em an evolution system} if the
following three conditions are satisfied
\begin{enumerate}
\item $U(t,t) = 1$, the identity operator, for all $t \in I$;

\item $U(t,t'')U(t'',t') = U(t,t')$ for $0\leq t'\leq t'' \leq t \in I$;

\item $U(t,t')$ is strongly continuous in  $t$, $t'$ for all $0\leq t'\leq t \in I$.
\end{enumerate}
\end{definition}

Informally, we shall say that the family of unbounded operators $L = (L(t))_{t \in I}$
{\em generates} the evolution system $U$ if $\partial_t U(t, s) \xi = L(t) U(t, s) \xi$ 
for all $t > s$ and $\xi$ in a suitable large subspace. We prefer not to give a formal
definition for what ``large'' means in this setting, as for the families
$L$ that we will consider, this will happen {\em everywhere.}

We next recall that uniform sectoriality implies generation of an evolution system 
\cite[page 212]{Lunardi}. (This is the ``uniform parabolic case,''
see also sections 5.6 and 5.7 in \cite{Pazy}.)
Let $\operatorname{arg} : \CC \smallsetminus (-\infty, 0] 
\to \CC$ be the imaginary part of the branch of $\log$ that satisfies $\log (1) = 0$.
For the rest of this paper, $I \subset [0, \infty)$ will be an interval
containing 0.

\begin{definition}\label{assumption.L(t)}
A family of operators $L = (L(t))_{t \in I}$, 
$L(t) : \cD(L(t))\subset
X\rightarrow X$, $t \in I$, will be called {\em uniformly
  sectorial} if the following conditions are satisfied:
\begin{enumerate}
\item The domains $\cD(L(t))=:\cD$ are independent of $t$ and dense in $X$;
\item $\cD$ can be endowed with a Banach space norm such that
the injection \ $\cD \hookrightarrow X$
is continuous and $I \ni t \to L(t) \in  \cL(\cD, X)$ is
uniformly H\"older continuous with exponent $\alpha\in(0,1]$.
\item There exist $\omega \in \RR,\theta \in
  (\pi/2,\pi),M>0$ such that, for any $t \in [0, T)$
\begin{equation*}
\left \{
\begin{array}{l}
  \ \rho(L(t))\supset S_{\theta,\omega} \ede \{\lambda \in \CC, \ \lambda \neq
  \omega,\, |\operatorname{arg}(\lambda-\omega)|<\theta\}, \\[2mm]
  \ \|R(\lambda,L(t))\| \leq \frac{M}{|\lambda-\omega|}, \quad \forall \lambda
  \in S_{\theta,\omega}.
\end{array}
\right.
\end{equation*}
\end{enumerate}
\end{definition}

For later use, we recall the following useful result
that applies to our setting, which is introduced in Section \ref{sec4}.
(See, for example, \cite[Corollary 6.1.8, page 219]{Lunardi},
for a proof.)

\begin{theorem}\label{thm.properties.evolution.system}
Suppose $L = (L(t))_{t \in I}$ is uniformly sectorial with
common domain $\cD$, then $L$ generates an evolution
system $U(t,r)$ 
{\rm (\ie $\partial_t U(t, s)\xi = L(t)U(t, s)\xi$ for all $\xi \in \cD$).} 
This evolution system is unique and, for any $0\leq s \leq t$, $s, t \in I$, 
we have the following
\begin{enumerate}
\item The functions
$$ \|U(t,s)\|_{X}, \quad (t-s)\|U(t,s)\|_{X,\cD} , \quad
  \|L(t)U(t,r)\|_{\cD,X}$$
are uniformly bounded for $t, s \in I$, $s \le t$; and
\item $\partial_s U(t,s) = - U(t,s)L(s).$
\end{enumerate}
\end{theorem}

We now return to the study of the IVP \eqref{eq.def.IVP}.  
We shall use the following notion of solution (see e.g. \cite[pages 123-124]{Lunardi}).

\begin{definition}\label{def.solutions}
Let $X$ be a Banach space, $h\in X$, and $f \in L^1((0,T), X)$.
\begin{enumerate}
\item By a {\em strong solution in $X$} of   \eqref{eq.def.IVP} on the interval 
$[0,T]$, we mean a function
\begin{equation}\label{eq.def.st}
    u \in \cC([0,T), X) \cap
    W^{1,1}((0,T),X)
\end{equation}
such that in $X$ $\partial_t u(t) = L(t) u(t) + f(t)$  for almost all
$0<t<T$, and \nolinebreak[4]
\mbox{$u(0) = h$.}

\item By a {\em classical solution in $X$} of \eqref{eq.def.IVP} on the interval 
$[0,T]$, we mean a function
\begin{equation}\label{eq.def.cs}
    u \in \cC([0,T], X) \cap
    \cC^1((0,T),X)\cap \cC((0,T), \cD(L(t)))
\end{equation}
such that in $X$ $\partial_t u(t) = L(t) u(t) + f(t)$ for $0<t<T$, and
$u(0) = h$.
\end{enumerate}
\end{definition}

Theorem \ref{thm.properties.evolution.system} shows that, if $f\equiv 0$ and 
$L(t)$ is uniformly sectorial, then the IVP \eqref{eq.def.IVP} has a unique strong 
and classical solution for all $h\in X$. 

We shall need the following results (see \cite{Amann, Lunardi, Pazy}).

\begin{lemma}\label{lemma.renorm}
Assume that $U(t,s)$, $s, t \in I$, $0 \le s \le t$,
is an evolution system that has exponential bounds in the sense
that there exists $\omega_U \in \RR$ and $M_U >0$, such that, for all 
$x \in X$ and all $s, t \in I$, 
$t\geq s\geq 0$,
\[ 
   \| U(t,s) x\| \leq  M_U\, e^{\omega_U (t-s)} \,\|x\|\,.
\]
Then there exists a uniformly equivalent family of
time-dependent norms $|||\cdot |||_t$, meaning 
\begin{equation*}
   C_U^{-1} \| x \| \le |||x |||_t \le C_U^{-1} \| x \|
\end{equation*}
for all $x \in X$ and all $t \in I$, such that, for all 
$t\geq s\geq 0$ and all $x \in X$,
\begin{equation*}
  ||| U(t,s)x |||_t \leq e^{\omega_U (t-s)} |||x |||_s\,.
\end{equation*}
\end{lemma}

\begin{proof} 
Set $V(t,s)=e^{-\omega (t-s)}U(t,s)$, then it is clear
that $V(t,s)$ is uniformly bounded by $M_U$. We define a new norm as
\begin{equation*} 
  |||x|||_s \ede \sup_{s \le t \le T} \|V(t,s)f\| \,.
\end{equation*}
From the first part, we then obtain $\|x\|\leq |||x|||_s \leq M_U \|x\|$, 
for all $s \in I$. Thus, for all $s \in I$, 
$|||\cdot|||_s$ is equivalent to $\|\cdot\|$ on $X$. Note that by our
definition
\begin{eqnarray*}
  ||| V(t,s)x |||_t  &  \seq  & \sup_{r\geq t} \|V(r,t)V(t,s)x\| \seq \sup_{r\geq
    t}\|V(r,s)x\|\\
  & \leq & \sup_{r\geq s} \|V(r,s)x\| \, =: \, |||x|||_s \, .
\end{eqnarray*}
Substituting $V(t,s)=e^{-\omega (t-s)}U(t,s)$, we obtain the desired
estimate.
\end{proof}

The first part of the following corollary is a very classical result
\cite{Amann, Lunardi, Pazy}.

\begin{corollary}\label{cor.lemma.renorm}
Assume that the family $L = (L(t)_{t \in I}$ of operators
on a Banach space  $(X,\|\cdot\|)$ is uniformly sectorial and let $U$ be the 
evolution system it generates. Then $U$ has exponential bounds, meaning that
there exists $\omega_U$, $M>0$, such that, for all $t\geq s\geq 0$ 
and all $x \in X$,
\[ 
   \| U(t,s) x\| \leq  M\, e^{\omega_U (t-s)} \,\|x\|\,.
\]
Also, there exists a uniformly equivalent family of
time-dependent norms $|||\cdot |||_t$ such that, for all 
$t\geq s\geq 0$ and all $x \in X$,
\begin{equation*}
  ||| U(t,s)x |||_t  \leq   e^{\omega_U (t-s)} |||x |||_s\,.
\end{equation*}
\end{corollary}

\begin{proof}
The first part is very well known \cite{Amann, Lunardi, Pazy}. 
The second part is also known (and it follows easily from Lemma \ref{lemma.renorm}).
\end{proof}

\section{Discretization and bootstrap error estimates}
\label{sec3}

Let $U(t, s)$ be an evolution system acting on some Banach
space $(X, \| \cdot \|)$, $s, t \in I$, $0 \le s \le t$, 
satisfying exponential bounds. In this section, we study 
the discretization error when 
we compress $U$ to a subspace $S \subset X$
and the bootstrap error
when we approximate $U$ with some other two-parameter
family of operators $K$. Since our results are for $s, t \le T$,
for some fixed $T>0$, there is no loss of generality to assume
that $I = [0, T]$, again, for some $T > 0$ that is fixed in this section.

Throughout this section, let $U(t, s)$ be an evolution system acting on some 
Banach space $X$, $s, t \in I$, $0 \le s \le t$, satisfying
exponential bounds, that is, such that there exist $\omega_U \in \RR$ and
$M_U > 0$ with the property that $\| U(t,s) x\| \leq  M_U\, e^{\omega_U (t-s)} 
\,\|x\|\,$ for all $x \in X$ and all $s, t \in I$, $0 \le s \le t$. Recall 
then from Lemma \ref{lemma.renorm} that there exist $C_U > 0$ and 
norms $|||\cdot |||_t$, $t \in I$, on $X$ such that
\begin{equation}\label{eq.equiv.norms}
  \begin{gathered}
   C_U^{-1} \| x \| \le |||x |||_t \le C_U^{-1} \| x \| \quad \mbox{and}\\
  ||| U(t,s)x |||_t \leq e^{\omega_U (t-s)} |||x |||_s\,.
  \end{gathered}
\end{equation}
for all $x \in X$ and all $s, t \in I$, $0 \le s \le t$.
A family of norms satisfying the first property of the above equation will 
be called a {\em uniformly equivalent family of
time-dependent norms} on $X$. Notice that there is {\em no additional bound} in front
of the exponential in the last estimate, and this is indeed crucial in our
error estimates below. The need for such estimates is one feature that is specific 
to time dependent equations. Below, $U$, $C_U$, and $\omega_U$ will always
be as in the above equation. We recall that evolution systems generated by
uniform parabolic (the case in the following sections) or uniform hyperbolic
generators will satisfy our assumptions \cite{Pazy}.

Let $X_s = X$, but with the norm $||| \cdot |||_s$. We let $|||T|||_{s, t} :=
\|T\|_{\cL(X_s, X_t)}$.


We shall need the following simple lemmas

\begin{lemma}\label{lemma.est.PU1}
We let $C_U$ and the norms $||| \cdot |||_t$ on $X$ be as in Equation
\ref{lemma.renorm}. Then, 
for all $Q \in \cL(X)$, we have $\| Q \|_{s, t} \, \le \, C_U^2 \|Q\|_{X}$.
\end{lemma}

The proof is immediate. We have stated this lemma only for reference
purposes.

\begin{lemma}\label{lemma.est.PU}
Let $V(t, s), G(t,s) \in \cL(X)$, $0 \le s \le t \le T$. Let us assume the 
following:
\begin{enumerate}
\item There exists $\omega \in \RR$ such that $|||V(t, s)|||_{s, t} \le 
e^{\omega (t-s)}$ for all $0 \le s \le t \le T$, where 
$|||\cdot |||_{s,t}$ is the operator norm $(X, |||\cdot |||_s) \to
(X, |||\cdot |||_t)$ for a family of norms $||| \cdot |||_t$, $0 \le t \le T$ on $X$ 
that is uniformly equivalent to $\| \cdot \|$.

\item There exist $\alpha \ge 1$ and $C_G > 0$ such that
$\| V(t, s) - G(t,s)\|_X \le C_G (t-s)^{\alpha}$ for all  $0 \le s \le t \le T$.
\end{enumerate}
Then there exists $\omega' \in \RR$ such that $|||G(t, s)|||_{s, t} \le 
e^{\omega' (t-s)}$ for all $0 \le s \le t \le T$.
\end{lemma}

\begin{proof}
We first notice that, by Lemma \ref{lemma.est.PU1},
we have $||| V(t, s) - G(t,s) |||_{s, t} \le 
C_U^2 C_G (t-s)$ for all $0 \le s \le t \le T$. Then, we notice that,
for large $\omega'$ fixed, we have
\begin{equation*}
   \sup_{0 \le s, t \le T} \frac{e^{\omega |t-s|} 
   + 2 C_U^2 C_G |t-s|^\alpha}{e^{\omega' |t-s|}} \le 1\,.
\end{equation*}
The result then follows from the triangle inequality.
\end{proof}

Notice that if we replace the condition $\| V(t, s) - G(t, s)\|_{X} 
\le C_K (t-s)$ with the condition 
$\| V(t, s) - G(t, s)\|_{X} \le C_K (t-s)^\alpha$, for some 
$\alpha < 1$, then, in general, the lemma will not be true anymore.
We are ready now to prove an error estimate for the spatial discretization

\begin{theorem}\label{thm.disc.error}
Let $I = [0,T]$ and $U(t, s)$, $0 \le s \le t \le T$ be an evolution system on 
a Banach space $X$ as in Equation \eqref{eq.equiv.norms}. 
There exist $C_U > 0$ with the following property.
Let $P : X \to S \subset X$ be a continuous linear projection and let
$C_P > 0$ be such that
\begin{equation}\label{eq.def.cond.com}
  \| (1- P) U(t, s) P\|_{X} \, \le \, C_P (t - s)
\end{equation}
for all $0 \le s \le t \le T$. Let $0 \le T_0 \le T$, $n \in \NN$, $\delta := T_0/n$.
Also, let $x_k \in X$ and $y_k \in S$ satisfy $x_{k+1} = U \big ( (k+1) \delta, k \delta \big) x_k$ 
and $y_{k+1} = P U \big ( (k+1) \delta, k \delta \big ) y_k$. Then there
exists $\omega \in \RR$ such that 
\begin{equation*}
 \|x_n - y_n \| \, \le \, C_U^2 e^{\omega T_0}   \,
 \big ( \, \|x_0 - y_0 \| +  T_0 \|z_0\| \, \big)  \,.
\end{equation*}
\end{theorem}

The bound $C_P$ does not appear in the error estimate, but
we stress that $\omega$ depends on $C_P$.

\begin{proof}
We let $\omega_U, C_U$, and the norms $||| \cdot |||_t$ be as in Equation
\ref{lemma.renorm}. The families of operators $K := U$
and $\widetilde K := (1-P) U P$ satisfy the assumptions of Lemma \ref{lemma.est.PU}.
That lemma then shows that there exists
$\omega \in \RR$ such that, for all $0 \le s \le t \le T$ 
\begin{equation}\label{eq.est.p1}
 \|(1- P) U(t, s) P\|_{s, t} \seq \|\widetilde{K}(t, s)\|_{s, t} \le 
e^{\omega(t-s)}
\end{equation}
(we have replaced $\omega'$ with $\omega$). By induction on $k$, we then 
obtain 
\begin{equation}\label{eq.est.p2}
   |||y_k|||_{k \delta} \le e^{ k \omega \delta} |||y_0|||_{0} \,.
\end{equation}

We may assume that $\omega \ge \omega_U$. 
Let us then prove by induction the estimate
\begin{equation}\label{eq.est.p3}
   |||x_k - y_k |||_{k\delta} \, \le \, e^{\omega k \delta}  \,
   \big ( \, |||x_0 - y_0 |||_0  +  k \delta |||y_0|||_0 \, \big)\,,
\end{equation}
for all $0 \le k \le n$. Indeed, it is true for $k =0$ 
(we even have equality in that case). Assume it next to be true for $k$,
and let us prove it for $(k+1)$. We have
\begin{multline*}
  ||| x_{k+1} - y_{k+1} |||_{(k+1)\delta} \  = \ 
  ||| U \big (   (k+1) \delta, k \delta \big) x_k 
  - P U \big (   (k+1) \delta, k \delta \big) y_k |||_{(k+1)\delta}\\
   \le \  ||| U \big (   (k+1) \delta, k \delta \big) (x_k 
  - y_k ) |||_{(k+1)\delta}
   +
  \  ||| (1 -P) U \big (   (k+1) \delta, k \delta \big)  y_k |||_{(k+1)\delta}
  \\
   \le \  ||| U \big (   (k+1) \delta, k \delta \big)|||_{ (k+1) \delta, k \delta} 
   ||| x_k  - y_k  |||_{k\delta}
   +
  \  ||| (1 -P) U \big (   (k+1) \delta, k \delta \big) P|||_{ (k+1) \delta, k \delta} 
     |||y_k |||_{k\delta}
  \\
   \le \  e^{\omega \delta} \big ( ||| x_k  - y_k |||_{k\delta} + 
   |||y_k|||_{k\delta} \, \big)\,,\\
   \le \  e^{\omega \delta} \Big ( e^{\omega k \delta}  \,
   \big ( \, |||x_0 - y_0 |||_0  +  k \delta |||y_0|||_0 \, \big) + 
   e^{ k \omega \delta} |||y_0|||_{0} \, \big)\,,\\
   \le \  e^{\omega_U (k+1) \delta} \Big ( |||x_0 - y_0 |||_0  
  +  (k+1) \delta |||y_0|||_0 \, \Big)\,,
\end{multline*}
where the last two inequalities are obtained, in order,
from the estimates \eqref{eq.est.p1}, \eqref{eq.est.p2}, and  \eqref{eq.est.p3} 
(for $k$, the induction hypothesis). This proves \eqref{eq.est.p3}
for all $k$. The result follows from this relation for $k = n$,
using also Lemma \ref{lemma.est.PU1}.
\end{proof}

\begin{remark}
We stress that the appearance of the factor $(t-s)$ in Equation \eqref{eq.def.cond.com}
is crucial and is a typical feature of the conditions needed for the error estimates
in our bootstrap method. This condition can be achieved if the
commutator $[P, L(t)] := PL(t) - L(t)P$ is bounded on $X$. In turn, 
if $L = \Delta$, for instance and $X = L^2(\bR^N)$, 
then we can construct a subspace $S$ with these properties using a periodic partition
of unity and GFEM discretization spaces. The constant $C_{T, P}$, on
the other hand, can account for the spatial discretization error.
\end{remark}

The last theorem is relevant if we know $U(t,s)$ explicitly. This is rarely the
case. Instead (and this is one of the reasons why we are writing this paper), 
we can usually approximate $U(t,s)$. A general example of how to do that will be 
given in Section \ref{sec5}. We keep the setting of the previous theorem.

\begin{theorem}\label{thm.b.error}
Let $V(t, s), G(t,s) \in \cL(X)$, $0 \le s \le t \le T$, and $C_G>0$ be
as in Lemma \ref{lemma.est.PU}, with the most important
estimate being $\| V(t, s) - G(t,s)\|_X \le C_G (t-s)^{\alpha}$.
Then 
\begin{enumerate}
\item  There exists $\omega' \in \RR$ such that $|||G(t, s)|||_{s, t} \le 
e^{\omega' (t-s)}$ for all $0 \le s \le t \le T$. 

\item There are $\omega' \in \RR$, $C_V, C_N > 0$ with the following property. 
Let $n \in \NN$, $0 \le T_0 \le T$, $\delta := T_0/n$. Let also $x_k \in X$ 
and $y_k \in V$ satisfy $x_{k+1} = U \big ( (k+1) \delta, k \delta \big) x_k$ 
and $y_{k+1} = G \big ((k+1) \delta, k \delta \big) y_k$. Then
\begin{equation*}
   \|x_n - y_n \| \, \le \, C_G e^{\omega' T_0}  \,
 \Big ( \, \|x_0 - y_0 \| + C_{N} \frac{T_0^{\alpha}}{n^{\alpha-1}} \|y_0\| \, \Big) \,.
\end{equation*}
\end{enumerate}
\end{theorem}


\begin{proof}
Let $C_1 > 0$ be such that $C_1^{-1} ||| \xi |||_t \le \| \xi \| \le C_1 ||| \xi |||_t$
for all $t \in [0, T]$ and all $\xi \in X$, which exists since we have assumed that
the norms $||| \cdot |||_t$ are {\em uniformly} equivalent to the norm $\| \cdot\|$.
Then Lemma \ref{lemma.est.PU1} gives that
$||| V(t, s) - G(t,s) |||_{s, t} \le C_2 (t-s)^{\alpha}$ for all $0 \le s \le t 
\le T$, where $C_2 := C_1^2 C_G$.

The existence of $\omega'$ is the content of \ref{lemma.est.PU}. By increasing 
$\omega$, if necessary, we can 
assume that $\omega' = \omega$ in what follows. We proceed as in the proof of
Theorem \ref{thm.disc.error}. First, we similarly obtain, by induction, that
\begin{equation}\label{eq.est.pp2}
   |||y_k|||_{k \delta} \le e^{ k \omega \delta} |||y_0|||_{0} \,.
\end{equation}
The result will then follow from the estimate
\begin{equation}\label{eq.est.pp3}
   |||x_k - y_k |||_{k\delta} \, \le \, e^{\omega k \delta}  \,
   \big ( \, |||x_0 - y_0 |||_0  +  C_2 k \delta^\alpha |||y_0|||_0 \, \big)\,,
\end{equation}
valid for all $0 \le k \le n$, which we prove again by induction on $k$. Indeed, 
the estimate it is true for $k =0$ 
(we even have equality in that case). Assume it next to be true for $k$,
and let us prove it for $(k+1)$. We have
\begin{multline*}
  ||| x_{k+1} - y_{k+1} |||_{(k+1)\delta} \  = \ 
  ||| V \big ( (k+1) \delta, k \delta \big ) x_k 
  - G \big ( (k+1) \delta, k \delta \big) y_k |||_{(k+1)\delta}\\
   \le \  ||| V \big (  (k+1) \delta, k \delta \big) (x_k 
  - y_k ) |||_{(k+1)\delta}
   +
  \  ||| \big [ V \big (  (k+1) \delta, k \delta \big) - 
  G \big (  (k+1) \delta, k \delta \big) \big ] y_k |||_{(k+1)\delta}
  \\
   \le \  ||| V \big (  (k+1) \delta, k \delta \big)|||_{ (k+1) \delta, k \delta} 
   ||| x_k  - y_k |||_{k\delta}
   +
  \  ||| V \big (  (k+1) \delta, k \delta \big) - 
  G \big ( (k+1) \delta, k \delta \big) |||_{ (k+1) \delta, k \delta} 
     |||y_k |||_{k\delta}
  \\
   \le \  e^{\omega \delta} \big ( ||| x_k  - y_k |||_{k\delta} + 
   C_2 \delta^\alpha |||y_k|||_{k\delta} \, \big)\,,\\
   \le \  e^{\omega \delta} \Big ( e^{\omega k \delta}  \,
   \big ( \, |||x_0 - y_0 |||_0  +  C_2 k \delta^\alpha |||y_0|||_0 \, \big) + 
   e^{ k \omega \delta} C_2 \delta^\alpha |||y_0|||_{0} \, \big)\,,\\
   \le \  e^{\omega (k+1) \delta} \Big ( |||x_0 - y_0 |||_0  
  +  C_2 (k+1) \delta^\alpha |||y_0|||_0 \, \Big)\,,
\end{multline*}
where the last two inequalities are obtained, in order,
from the estimates $\| V(t, s) - G(t,s)\|_X \le C_G (t-s)^{\alpha}$, 
\eqref{eq.est.pp2}, and  \eqref{eq.est.pp3} 
(for $k$, the induction hypothesis). This proves \eqref{eq.est.p3}
for all $k$. The result follows from this relation for $k = n$,
using also Lemma \ref{lemma.est.PU1}.
\end{proof}

Since the first two conditions of the above theorem are automatically satisfied
by an evolution system, we obtain the following result.

\begin{corollary}\label{cor.b.error}
Let $U(t, s)$ be an evolution system on $X$, $0 \le s \le t \le T$,
and $G(t,s) \in \cL(X)$. 
Assume that there exist $\alpha \ge 1$ and $C_{G} > 0$ such that
$\| V(t, s) - G(t,s)\|_X \le C_{G} (t-s)^{\alpha}$ for all  $0 \le s \le t \le T$.
Then there is $C_{U, G, T} > 0$ with the following property. Let $n \in \NN$, 
$\delta := T/n$, $y_k \in V$ satisfy  $y_{k+1} = G \big ((k+1) \delta, k \delta \big) y_k$. 
Then
\begin{equation*}
   \|U(T, 0)y_0 - y_n \| \, \le \, C_{U, G, T}\, n^{1-\alpha}\, \|y_0\|  \,.
\end{equation*}
\end{corollary}

Here, of course, $C_{U, G, T}$ is independent of $n$ and $y_0$. In particular,

\begin{corollary}\label{cor.b.error2}
Using the notation of Corollary \ref{cor.b.error}, we have that,
for any $n \in \NN$, 
\begin{equation*}
   \Big \| \, U(T, 0) \, - \, \prod_{k=0}^{n-1}\, G \Big ( \frac{(k+1)T}{n}, 
   \frac{kT}{n} \Big )  \, \Big \| \ \le \ \frac{C_{U, G, T}} {n^{\alpha-1}}\,   \,.
\end{equation*}
\end{corollary}

See \cite{GlowinskiLions, schwabFinBook, Thomee, Yuwen, Samarskii, LarssonThomee}
for some general results on evolution equations that put our results into perspective.

\section{Analytic semigroups and Duhamel's formula} 
\label{sec4}

In this section, we introduce the class of uniformly
strongly elliptic operators that we study and we particularize
to them the theory recalled in Section \ref{sec2}. These
operators are particularly well suited to study via perturbative
expansions. In particular, in this section, using the theory
of analytic semigroups, we carefully check that all the integrals
appearing in Duhamel's formula and in perturbative series expansions
are well defined.

\subsection{Properties of the class $\bL_\gamma$}

Since the dimension $N$ is fixed throughout the paper, we will usually
write $W^{r,p}$ for $W^{r,p}(\RR^N)$. Similarly, we shall often write
$L^p$ instead of $L^p(\RR^N)$. When $1<p<\infty$, the dual of
$W^{r,p}$ is the Sobolev space $W^{-r,p'}$ with $1/p+1/p'=1$.

\begin{definition}\label{def.Lgamma}
A function is called {\em totally bounded} if itself and
all its derivatives are bounded. The set of totally bounded functions
defined on a set $\Omega \subset \bR^N$ will be denoted by $\cC^\infty_b(\Omega)$.
Let $I \subset [0, \infty)$ be an interval containing 0.
Let $\bL$ be the set of second-order differential operators 
$L = (L(t))_{t \in I}$
of the form
\begin{equation}\label{eq.L2}
    L(t) \seq \sum_{i,j=1}^N a_{ij}(t,x) \D_i \D_j + \sum_{k=1}^N
    b_k(t,x)\D_k + c(t,x),
\end{equation}
where the matrix $[a_{ij}]$ is symmetric and $a_{ij}, b_k, c \in
\cC^\infty_b( I \times\RR^N)$ are {\em real valued}.
Let $\bL_\gamma$ be the subset of operators $L \in \bL$ satisfying the
uniformly strong ellipticity condition \eqref{eq.uniformly.elliptic}
with given ellipticity constant $\gamma$.
\end{definition}

We utilize symbol calculus for pseudo-differential operators ($\Psi$DOs for short) 
to establish several results. We begin by recalling some basic facts about $\Psi$DOs. 
(See \cite{Tay, TayPDEII, Wong} for the definition and basic properties of 
pseudodifferential operators.)

We deal only with classical symbols in H\"ormander's class $S^m_{1,0}$, $m\in \RR$, 
and denote the symbol of a pseudo-differential operator $P$ by $\sigma(P)$ with 
$\sigma_0(P)$ its principal symbol. Conversely, 
given a symbol in $S^m_{1,0}$, we denote the associated pseudo-differential operator with $P=\sigma(x,D)$,  $D=\frac{1}{i}\partial$. We recall that any operator  with symbol in 
$S^{-\infty}=\bigcap_{m\in \RR} 
S^m_{1,0}$ is a smoothing operator.
We denote with  $\Psi^m_{1,0}$ the space of operators with symbols in $S^m_{1,0}$.
Every $\Psi$DO has distributional kernel $\sigma(x,D)(x,y)$ by the Schwartz Kernel Theorem (see
e.g. \cite{TayPDEI}). We will need to deal only with integral operators with smooth kernels.

\medskip

\noindent {\bf Notation:} {\em If an operator $T$ has smooth kernel, we will 
denote it by  $T(x,y)$.}

\medskip

If $P=\sigma(x,D)$ is smoothing, then there is a one-to-one correspondence between 
the symbol and the kernel:
\begin{equation*}
    \sigma(x, D) (x,y) =
        (\cF_2^{-1}\sigma)(x, x-y),
\end{equation*}
where $\cF_2$ the Fourier transform in the second variable of a function of two 
variables. We will also use the standard fact that multiplication with a smoothing 
symbol is continuous on any symbol class.

We recall that elliptic $\Psi$DOs in $\Psi^m_{1,0}$, $m\in \ZZ$, in particular
elements of $\LL_\gamma\subset \Psi^2_{1,0}$, generate equivalent
norms in Sobolev spaces \cite{MN}. This is a general fact that holds in the 
greater generality of manifolds with bounded geometry \cite{AGN1, AmannParab1, CGT, MN}.
In particular, we have  the following result.

\begin{corollary}\label{cor.norm.equiv2}
Suppose $L = (L(t))_{t \in I} \in \bL_\gamma$, $1< p < \infty$, and $m\in \ZZ_+$. 
Then the following two norms are
equivalent
\begin{equation}\label{norm.equivalence2}
  \|u\|_{W^{2m,p}} \sim \|u\|_{L^p} + \|L^m(t)u\|_{L^p},
\end{equation}
with constants that are uniform in $t \in I$.
\end{corollary}

Next we show that if $L = (L(t))_{t \in I} \in \bL_\gamma$, then $L(t)$ is H\"{o}lder
continuous in $t$, and sectorial for each $t\in I$ between the Sobolev spaces
$W^{2k+2,p}$ and $W^{2k,p}$, $1<p<\infty$, for each $k\in \ZZ_+$.

These properties in turn give the
needed mapping bounds for the evolution system discussed in Subsection
\ref{evolutionsystem}. (See \cite{Amann,Lunardi,Pazy} for instance.)
Below, $\cL(X_1, X_2)$ denotes the space of all bounded linear operators on $X_1 \to
X_2$ for two normed spaces $X_1$ and $X_2$, and we  write $\cL(X) = \cL(X, X)$.
We let $\| \cdot \|_{X_1, X_2}$ and $\| \cdot \|_{X}$ denote the
corresponding norms.

An immediate consequence of the definition of the space
$\bL_\gamma$ (Definition \ref{def.Lgamma}) gives that the function $I \ni t \to 
L(t) \in \cL(W^{k+2,p}, W^{k,p})$
is uniformly Lipschitz continuous. Furthermore, for each 
$t\in I$, $L(t):W^{2,p} \to L^p$, $1<p<\infty$, is a sectorial operator 
(see \cite[page 73]{Lunardi} for a proof).

This result readily generalizes to any $k\in \ZZ_+$. We sketch below 
a proof for completeness, but, first, let us recall
that we have the following well-known fact \cite{Amann, Lunardi, Pazy}.

\begin{lemma}\label{lemma.resolvent.set}
If $L = (L(t))_{t \in I} \in \bL_\gamma$, then for each $t \in I$ and $k$, $L(t)$
defines a continuous map $W^{2k+2,p} \to W^{2k,p}$ with the property
that the the resolvent set of $L(t)$ contains a half plane $\{\lambda
\in \bC ,\ \operatorname{Re} \lambda \geq \omega\}$.
\end{lemma}

This gives then the following result.

\begin{proposition}\label{prop.sectorial}
If $L = (L(t))_{t \in I} \in \bL_\gamma$, then for each $t\in I$ and $k$, the operator
$L(t): W^{2k+2,p} \to W^{2k,p}$ is sectorial.
\end{proposition}

\begin{proof}
We fix $t=t_0$ and simply write $L_0=L(t_0)$.
For any $u\in W^{2k,p}$ and $\lambda\in \rho(L_{0})$,  Lemma
\ref{lemma.resolvent.set} gives $R(\lambda,L_{0})u \in W^{2k,p}$. Next,
using the norm equivalence \eqref{norm.equivalence2} twice, the fact that $L(t)$ 
is sectorial, and standard properties of the resolvent,  we obtain
\begin{equation*}
\begin{split}
  \|\lambda R(\lambda, L_{0})u\|_{W^{2k,p}} & \leq C(\|\lambda
  R(\lambda,L_{0})u\|_{L^p} + \|\lambda L_{0}^k
  R(\lambda,L_{0})u\|_{L^p})\\
  & \leq C(\|u\|_{L^p}+\|L_{0}^ku\|_{L^p}) \leq C\|u\|_{W^{2k,p}},
\end{split}
\end{equation*}
with $C$ independent of $\lambda$. Lemma \ref{lemma.resolvent.set} and Proposition
\ref{condition.on.sectorial} then imply  that  $L_{0}:W^{2k+2,p} \to W^{2k,p}$ is sectorial.
\end{proof}

Recall that, by Theorem \ref{thm.properties.evolution.system}, if $f\equiv 0$ and 
$L(t)$ is uniformly sectorial, then the IVP \eqref{eq.def.IVP} has a unique strong 
and classical solution for all $h\in X$. In particular, if $L \in \mathbb{L}_\gamma$, 
we have well-posedness in $W^{k,p}$, $k\geq 0$, $1<p<\infty$
for our IVP, Equation \eqref{eq.def.IVP}.

Proposition \ref{prop.sectorial} and the properties of $L_\gamma$ 
show that any $L\in \bL_\gamma$ is a uniformly sectorial operators on $W^{2k,p}$. We 
use here that all bounds on the operator norm of $L(t)$ are uniform in $t\in [0,T]$ 
for fixed $0<T<\infty$. By duality and interpolation, we can obtain mapping properties 
between fractional Sobolev spaces $W^{s,p}$.

\begin{corollary}\label{mapping.property.U}
Suppose $L = (L(t))_{t \in I} \in \bL_\gamma$. Then $L$ generates an evolution system $U$ 
in $W^{s,p}$, $s \ge 0$, $1<p<\infty$,  such that the functions
\begin{equation*}
  \|U(t,t')\|_{W^{s,p},W^{s,p}} , \quad
  \|L(t)U(t,t')\|_{W^{s+2,p}, W^{s,p}} , \quad
  (t - t') \|U(t,t')\|_{W^{s,p}, W^{s+2,p}} 
\end{equation*}
are uniformly bounded for $t, t' \in I$, $0\leq t'\leq t$.
\end{corollary}

From Corollary \ref{mapping.property.U}, the fact that $L$ is Lipschitz and $U$ 
is bounded uniformly in time on $I$ as elements of $\cL(W^{s+2,p},W^{s,p})$ 
implies the following.

\begin{corollary}\label{lemma.cont}
Given $s\geq 0$, $1<p<\infty$, for any $t, t' \in I$, $0\leq t'\leq t$,
\begin{equation*}
  \|U(t,t') -  1\|_{W^{s+2,p}, W^{s,p}} 
  \, \leq \, C \, |t-r|,
\end{equation*}
with $C$ independent of $t, t' \in I$, $t' \le t$. In particular,
\begin{equation*}
  [t', \infty) \cap I \ni t \rightarrow U(t,t')\in
    \cL(W^{k+2,p}, W^{k,p})
\end{equation*}
defines a Lipschitz continuous map.
\end{corollary}

For the applications we have in mind, the initial data $h$ may not be integrable.
An example is provided by the payoff function of a European call option. To include 
such cases, we therefore introduce {\em exponentially weighted Sobolev spaces}. Given
a fixed point $w\in \RR^N$, we set
\begin{equation}\label{eq.def.weight2}
    \jap{x}_{w} := \jap{x-w} = (1 + |x - w|^2)^{1/2},
\end{equation}
with $\<,\>$ the Japanese bracket. For notational ease, we denote 
$\rho_{a}(x)=e^{a \jap{x}_w}$, with $w$ implicit. Then, for $k \in \ZZ_+$,
$a\in \RR$, $1<p<\infty$,
\begin{multline}\label{def.w.S}
   W_{a, w}^{k, p} (\RR^N) :=
    \{u:\RR^N \to \RR,\ \partial^\alpha \big( \rho_a
    u \big) \in L^p(\RR^N)\ |\alpha| \le k \},
\end{multline}
with norm
\begin{equation*}
    \|u\|_{W_{a, w}^{k, p}}^p := \| \rho_a u\|_{W^{k,p}}^p =
    \sum_{|\alpha| \le k} \|\partial^\alpha_\xi \big( \rho_a
    u \big)\|_{L^p}^p.
\end{equation*}
Weighted fractional spaces $W^{s,p}_{a,w}$, $s\geq 0$,  can then be defined by 
interpolation, and negative spaces by duality $W^{-s,p}_{a,w}=(W^{s,p'}_{-a,w})'$, 
with $p'$ the conjugate exponent to $p$.
The parameter $w$ will be called {\em the weight center}. Different choices of $w$ 
give equivalent norms and we also write $W^{s,p}_{a,w}=W^{s,p}_a$, since this vector 
space does not depend on $w$.

Recall that $\rho_{a}(x) := e^{a\jap{x}_z}$.
We study the operator $L(t)$ on the weighted spaces by conjugation. To this end, we 
define the operator $L_a(t) : = \rho_a L(t) \rho_a^{-1}$ and observe that  
$L: \, W^{s,p}_{a,w} \to W^{s,p}_{a,w}$ if and only if $L_a:\, W^{s,p} \to W^{r,p}$.

\begin{lemma}\label{lemma.same.class}
If $L = (L(t))_{t \in I} \in \bL_\gamma$ and $a \in \RR$, then 
$ \rho_aL\rho^{-1}_a = (L_a(t))\in \bL_\gamma$.
\end{lemma}

\begin{proof}
We compute $L_a(t)-L(t)$:
\begin{eqnarray*}
  L_a(t)-L(t) =
    \rho^{-1}\Big [\sum 2 a_{ij} \partial_i
    \rho \partial_j + \big (\sum_{i,j} \partial_i \partial_j
    \rho + \sum b_i \partial_i \rho \big)
    \Big] u.
\end{eqnarray*}
Since $\jap{x}_w$ has bounded derivatives, $L_a(t)-L(t))$ is a first order differential operator the
coefficients of which are smooth with all their derivatives uniformly
bounded. Hence $L_a(t)$ satisfies the same assumptions as $L(t)$.
\end{proof}

\begin{remark}\label{rem.cont}
By Lemma \ref{lemma.same.class}, we can then reduce to study the case $a=0$. Therefore, for instance,
$L(t) : W^{s+2,p}_{a,z} \to W^{s,p}_{a, z}$ is well defined and continuous for any $a$, since this
is true for $a=0$.
More generally, the results of Corollary \ref{lemma.cont} apply with
$W^{k,p}$ replaced by $W^{s,p}_{a,w}$ for any $w$ and $a$.
\end{remark}

See also \cite{AmannMaxReg, AmannFunctSpaces, AmannParab1, MN} for further, related results.

\subsection{Analytic semigroups}

In the construction of the asymptotic expansion for $U(t,0)$ in
Section  \ref{sec4} below, we will need smoothing properties
for the semigroup generated by a certain time-independent operator $L_0$
related to $L$. To this end, we recall needed basic facts
about analytic semigroups. (We refer again to \cite{Amann, Lunardi,
Pazy} for a more complete treatment.)

If $A$ is sectorial, then it generates an analytic semigroup.
One of the most important properties of analytic semigroups is the
following smoothing properties, which we state only for
time-independent operators $L_0$ in
the class $\LL_\gamma$ acting on the Sobolev space $W^{k,p}_{z,a}$. A
general proof can be found in \cite{Pazy} for instance.

\begin{proposition}\label{mapping.L0}
Let  $L_0 \in \bL_\gamma$ be time independent. Then
$e^{tL_0}$ is continuous on $[0,T]$, for any given
$0<T<\infty$,  and for $t>0$,
\begin{equation} \label{eq.mapping.L0}
  \|e^{tL_0}f\|_{W^{r,p}_{z,a}} \leq
  C(r,s)\, t^{(s-r)/2}\|f\|_{W^{s,p}_{z,a}},\quad r\geq s,
\end{equation}
with $C(r,s)$ independent of $t$.
\end{proposition}

An immediate consequence of the above result is the following corollary.

\begin{corollary}\label{cor.cont}\
Let $s,r\in \RR$ be arbitrary and $L_0 \in \bL_\gamma$ be time
independent. Then, the map
\begin{equation*}
  (0, T) \ni t \to e^{tL_0} \in \cL(W^{s, p}_{a, z},W^{r, p}_{a, z})
\end{equation*}
is infinitely many times differentiable.
\end{corollary}

We assume that we are given a {\em time independent} operator $L_0\in \bL_\gamma$
for a fixed $\gamma >0$ and let  $L\in \bL_\gamma$. We write
\begin{equation}\label{eq.splitting}
  L(t) = L_0 + V(t),
\end{equation}
and study the classical question of relating the evolution system $U(t,s)$ generated by $L$ to the
semigroup $e^{tL_0}$ generated by $L_0$ \cite{Amann,Lunardi}.
\medskip

\noindent {\bf Notation:}
{\em We denote the solution operator of the IVP \eqref{eq.def.IVP} for $s = 0$,
that is, $U(t,0) = U^L(t, 0)$, simply by $U(t)$, a one-parameter family of linear
operators.}

\subsection{Duhamel's formula}

We write the general IVP for  $L_0$ as
\begin{equation}\label{eq.def.IVP.L0}
\begin{cases}
  \D_t u(t, x) - L_0u(t, x) = f(t,x), & \hspace{1.0cm} \mbox{in}\;
  (0,\infty) \times \bR^N \\
  u(0, x) = h(x), \ \ & \hspace{1.0cm} \mbox{on}\; \{0\}\times
  \bR^N\,,
\end{cases}
\end{equation}

\begin{lemma} \label{lemma.Duhamel1}
Let $h\in L^p$, $1<p<\infty$, and let $0<T\leq \infty$.
If $f\in L^1((0,T),L^p) \bigcap \cC((0,T],L^p)$  and $u$ is the
unique strong  solution to \eqref{eq.def.IVP.L0} on $[0,T]$, then $u$ is given by
\begin{equation*}
  u(t,x) = e^{tL_0} h + \int_0^t e^{(t-\tau)L_0}h(\tau)d\tau, \quad
  0\leq t \leq 1.
\end{equation*}
If $f$ satisfies in addition $f\in C^\alpha((0,T);L^p)$ for some $0<\alpha$, 
then \eqref{eq.def.IVP.L0} has a unique strong solution $u$.
\label{lemma.Duhamel.i}
\end{lemma}

\begin{proof} This proof is standard (see e.g. \cite[Theorem 2.9, page 107,Corollary 3.3, page 123]{Pazy}), 
using the fact that $L_0$ generates an analytic semigroup.
\end{proof}

We obtain the following consequence.

\begin{corollary}\label{cor.Duhamel2}
Let $u(t)$ be the unique classical solution of the IVP \eqref{eq.def.IVP}
with $s =0$. Then $u$ solves the Volterra-type equation
\begin{equation}\label{eq.duhamel}
  u(t) = U(t) h = e^{tL_0}  h + \int_0^t e^{(t-\tau)L_0}
  V(\tau)\, u(\tau)d\, \tau
\end{equation}
where $V$ is given  in \eqref{eq.splitting}.
\label{lemma.Duhamel.ii}
\end{corollary}

\begin{proof} By density, we first assume that
$h \in W^{2,p}$, and observe that, formally, the solution the IVP 
\eqref{eq.def.IVP} satisfied \eqref{eq.def.IVP.L0} with
\begin{equation*}
  f(t) = Vu(t,x) = (L(t)-L_0)u(t,x) = u_t(t,x)-L_0U(t)f.
\end{equation*}
Since the solution operator $U(t)$ of the IVP \eqref{eq.def.IVP} satisfies
$U(t): W^{2,p}\rightarrow W^{2,p}$ as a bounded operator that is strongly 
continuous  for $t\geq 0$ and continuously differentiable for $t>0$,
$L_0U(t)f\in L^p$ has this regularity.  But
$u_t\in L^p$ share the same regularity, given that $u$ is a classical solution.
 Therefore, by  the first part and the uniqueness of classical solutions,
$u$ must agree with  \eqref{eq.duhamel}.
Next, given $h\in L^p$, there exists $h_n \in W^{2,p}$, $h_n \to h$ in $L^p$. Let
$u_n$ be the strong solution with $u_n(0)=h_n$. Then $u_n$ satisfies
\[
     u_n(t) = U(t)h_n = e^{tL_0}f + \int_0^t e^{(t-\tau)L_0}\,
  V(\tau)\, u_n(\tau)\, d\tau.
\]
We would like to pass to the limit $n \to \infty$ on the right-hand
side of the expression above. In order to do so, we will use
the mapping properties of the semigroup $e^{tL_0}$ (Proposition
\ref{mapping.L0}) and of the evolution system $U(t)$ (Corollary
\ref{mapping.property.U}) to show that the integral is the action of a
continuous operator on $L^p$. Indeed,
\begin{align}
 & \Big \|\int_0^t e^{(t-\tau)L_0}V(\tau)U(\tau)d\tau \Big \|_{L^p}  
 \nonumber\\
 &\qquad \leq
\int_0^t\|e^{(t-\tau)L_0}\|_{W^{-1,p}, L^p} \|V(\tau)\|_{W^{1,p}, W^{-1,p}} 
\|U(\tau)\|_{L^p, W^{1,p}} d\tau \nonumber\\
 & \qquad \qquad \qquad \leq \int_0^t
\frac{1}{\sqrt{t-\tau}}\frac{1}{\sqrt{\tau}}d\tau<\infty
\end{align}
The proof is complete.
\end{proof}

\begin{remark}
Solutions to the Volterra equation \eqref{eq.duhamel} are called {\em mild} solutions.
Under the assumptions of the Lemma, classical and strong solutions of \eqref{eq.def.IVP.L0} 
are mild solutions, which are in particular unique.
In fact, if $f$ is locally H\"older's continuous in time, mild solutions are classical 
solutions  \eqref{eq.def.IVP.L0} \cite[Theorem 3.2, page 111]{Pazy}.
\end{remark}

Using this lemma, we can generalize the bounds contained in Corollary
\ref{mapping.property.U}.

\begin{lemma}[Mapping properties of $U(t,r)$]
\label{mapping.property.U(t,r)}
Let $U$,  be the evolution system generated by the
operator $L\in \LL_\gamma$ on $[0,T]$.  For any  $0\leq k\leq r$, $a\in \RR$,
$1<p<\infty$, $U(t_1,t_2): W^{k,p}_{z,a} \to W^{r,p}_{z,a}$ if $t_2<t_1$, and
there exists $C>0$ independent of $t_1$, $t_2$ such that
\begin{equation*}
  \|U(t_1,t_2)\|_{W_{a, z}^{k,p}, W_{a, z}^{r,p}}\leq
  C(t_1-t_2)^{(k-r)/2}.
\end{equation*}
\end{lemma}

\begin{proof}
We set  $a=0$ by Lemma \ref{lemma.same.class} and,  as $p$ is fixed,  
write  $W^k=W^{k,p}$. In particular, $\| T \|_{W^k, W^m} =
\|T\|_{\cL(W^{k,p}, W^{m,p})}$.
We temporarily
assume that $k\leq r< k+2$.
From \eqref{eq.duhamel}, for any $0\leq t_2\leq t_1\leq 1$ and any $h\in W^k$,
\[
   U(t_1,t_2) h = e^{(t_1-t_2) L_0} \, g +
         \int_{0}^{t_1-t_2} e^{(t_1-t_2-\tau) L_0} \,V(\tau)\,
       U(t_2+\tau,t_2) h\,d\tau.
\]
From Corollary
\ref{mapping.property.U} and Proposition
\ref{mapping.L0}, it follows  that
\begin{equation*}
\begin{split}
  & \|U(t_1,t_2)\|_{W^{k },  W^{r }}\leq
  \|e^{(t_1-t_2)L_0}\|_{W^{k } ,  W^{r }}\\
  & +\int_{0}^{\frac{t_1-t_2}{2}}
  \|e^{(t_1-t_2-\tau)L_0}\|_{W^{k-2} , 
    W^{r}}\|V(\tau)\|_{W^{k} , 
    W^{k-2}}\|U(\tau+t_2,t_2)\|_{W^{k}, W^{k}}\,d\tau\\
  & + \int_{\frac{t_1-t_2}{2}}^{t_1-t_2}
  \|e^{(t_1-t_2-\tau)L_0}\|_{W^{k}\rightarrow
    W^{r}}\|V(\tau)\|_{W^{k+2}\rightarrow
    W^{k}}\|U(\tau+t_2,t_2)\|_{W^{k}\rightarrow W^{k+2}}\,d\tau\\
  &\leq C \left
  ((t_1-t_2)^{\frac{k-r}{2}}+\int_{0}^{\frac{t_1-t_2}{2}}
  (t_1-t_2-\tau)^{\frac{k-2-r}{2}}\, d\tau \right .\\
  & \left . \qquad \qquad \qquad + \int_{\frac{t_1-t_2}{2}}^{t_1-t_2}
   (t_1-t_2-\tau)^{\frac{k-r}{2}} \,\tau^{-1} \,d\tau
  \right) \leq C(t_1-t_2)^{\frac{k-r}{2}},
\end{split}
\end{equation*}
using that \ $0<(r-k)/2 <1$\  by hypothesis.
Next,  let $\delta=\frac{r-k}{m}$, where $m$ is an
integer and $m>\frac{r-k}{2}$. Then the previous estimate gives
\begin{equation*}
  \Big \|U \Big (  t_1-(j-1)\frac{t_1-t_2}{m},t-j\frac{t_1-t_2}{m} \Big) 
  \Big \|_{W^{k+(j-1)\delta } 
  \rightarrow W^{k+j\delta }} \leq C\left (\frac{t_1-t_2}{m} \right
  )^{\frac{k-r}{2m}},
\end{equation*}
for $j=1,2,\cdots , m$. Therefore,
\begin{equation*}
  \|U(t_1,t_2)\|_{W^{k }\rightarrow W^{r }}\leq C\left
  (\frac{t_1-t_2}{m} \right
  )^{m \frac{k-r}{2m}}=C(t_1-t_2)^{(k-r)/2},
\end{equation*}
where $C$ depends on $k,r,p$, but not on $t_1$, $t_2$.
\end{proof}

In particular, the solution operator $U(t)$ of \eqref{eq.def.IVP} is
smoothing of infinite order on any positive Sobolev space $W^{k,p}_{z,a}$ (in fact, by duality,
on any Sobolev space) if $t>0$, as it is the case for $e^{t L_0}$.

\begin{corollary}\label{cor.cont.U}
If $L(t)\in \bL_{\gamma}$, and $U(t,r),t\geq r\geq0$ is the
resulting evolution system, then
\begin{equation*}
  (0,+\infty)\ni t \rightarrow U(t,r)\in \cB(W^{s, p}_{a, z}, W^{m, p}_{a, z})
\end{equation*}
is infinitely many times differentiable for any $s$ and $m$,
$1<p<\infty$, $a\in \RR$, and any $z\in \RR^3$.
\end{corollary}

We omit the proof as it is very similar to that of Corollary \ref{cor.cont}.
Another consequence of Lemma \ref{mapping.property.U(t,r)} is that the distributional 
kernel of the operator $U$, the  {\em Green's function} or {\em fundamental solution} 
for \eqref{eq.def.IVP}, $\cG_t^L \in \cC^\infty(\bR^N \times \bR^N)$.
In fact, $\cG^L_t$ is given by
\begin{equation*}
  \cG^L_t(x,y)= \jap{\delta_x,U(t)\delta_y}\,,
\end{equation*}
where $\jap{\cdot,\cdot}$ is the duality pairing between
$\cC^\infty(\RR^N)$ and compactly supported distributions, 
and where $\delta_z$ is the Dirac delta centered
at $z$. One of the goals of this work is to obtain
{\em explicit} approximations of $\cG^L_t(x,y)$ with good error bounds.

\begin{remark}\label{rem.simplex}
For each $k\in \ZZ_+$, we let
\begin{multline*}
  \Sigma_k := \{\tau = (\tau_0, \tau_1, \ldots, \tau_{k}) \in
  \RR^{k+1},\ \tau_j \ge 0, \sum \tau_j=1\} \\
  \simeq \{\s=(\s_1, \ldots, \s_k) \in \RR^{k},\ 1 \ge \s_1 \ge \s_2
  \ge \ldots \s_{k-1} \ge \s_{k} \ge 0,
\}
\end{multline*}
the {\em standard unit simplex} of dimension $k$. The bijection
above is given by $\s_j = \tau_j+\tau_{j+1}+\ldots+\tau_{k}$. Using
this bijection and the notation $d \sigma := d\s_{k} \ldots d\s_1$, 
for any operator-valued function $F$ on $\RR^N$, we have
\begin{equation*}
  \int_{\Sigma_k} F(\tau)d\tau \seq
  \int_{0}^{1}\!\!\int_{0}^{\s_1}\!\! \ldots \!\!
  \int_{0}^{\s_{k-1}}\!\! F(1-\s_1, \s_1- \s_2, \ldots, \s_{k-1} -
  \s_{k}, \s_k)\, d\sigma
\end{equation*}

\end{remark}

We begin with a preliminary technical lemma.

\begin{lemma}\label{lemma.perturbative}
Let $L_j \in \bL_\gamma$ and let $V_j$ be
such that $e^{-b_j\<x\>}
V_j\in \bL$,  $j=1,\ldots,k$, for some $b=(b_1,\ldots,b_k)\in
\RR_+^k$, $k \in \ZZ_+$. Then the function
\begin{equation*}
    \Phi(\tau) = e^{\tau_0L_0}V_1 e^{\tau_1L_1} \dots e^{\tau_{k-1} L_{k-1}}
    V_k E(\tau_k), \qquad \tau \in \Sigma_k,
\end{equation*}
 where either $E(\tau_k)=e^{\tau_kL_k}$ or
$E(\tau_k)=U(\tau_k) =U(\tau_k, 0)$,
defines a continuous function $\Phi : \Sigma_k \to \cB(W^{s, p}_{a,
  z}(\RR^N), W^{r,p}_ {a- |b|,z}(\RR^N))$ for any $a, r, s \in \RR$,
and $1<p<\infty$.
\end{lemma}

\begin{proof}
It suffices to prove that $\Phi$ is
continuous on each of the sets $\cV_j :=\{ \tau_j > 1/(k+2)\}$,
$j=0,\dots,k$, since they cover $\Sigma_k$. It also suffices to consider
the case $r \ge s$.
For $0 \leq j < k$, without loss of generality, we can assume that, in fact,
$j=0$ and prove continuity on the set $\cV_0$. The case $j=k$ will be
discussed below.

We define recursively numbers  $c_j = c_{j+1} -b_{j+1}$, $c_k = a$,
$r_j = r_{j+1} - 4$, $r_k = s$ for $j=1,\ldots,k-1$.
By the assumption on the $V_j's$ and thanks to  Proposition
 \ref{mapping.L0}  and Corollary
\ref{cor.cont},  each of the functions
\begin{align}
    [0,\infty)\ni \tau_j &\to V_j e^{\tau_jL_j} \in
    \cB(W^{r_j+4, p}_{c_j},W^{r_j, p}_{c_j-b_j}), \quad 1 \le j < k,
     \nonumber \\
    [0,\infty) \ni \tau_k &\to V_k E(\tau_k) \in
    \cB(W^{r_k+4, p}_{c_k},W^{r_k, p}_{c_k-b_k}), \nonumber
\end{align}
is continuous, and hence their  composition is continuous as a bounded
map $W^{s, p}_{a} \to W^{s-4k, p}_{a-|b|}$.
Since $e^{\tau_0 L_0}$ is continuous as a bounded operator
$W^{s-4k, p}_{a-|b|} \to W^{r, p}_{a-|b|}$ if $\tau_0>0$
thanks to Corollary \ref{cor.cont}, we conclude that the map
\begin{equation*}
     \cV_0 \ni \tau \to \Psi(\tau)=
    e^{\tau_0 L_0} V_1 e^{\tau_1L_1} ... V_k e^{\tau_kL_k}\in
    \cB(W^{s, p}_{a},W^{r, p}_{a-|b|})
\end{equation*}
is continuous.

For $\tau \in \cV_k$, we use instead Proposition
\ref{mapping.L0} to show continuity of  $e^{\tau_0 L_0}$ in $W^{r,p}_{a-|b|}$ for $t_0\in
[0,+\infty)$,   and Corollary \ref{cor.cont.U}  to show continuity of the map
\begin{equation*}
    (0,\infty) \ni \tau_k \to E(\tau_k) \in
    \cB(W^{s-4k,p}_{a-|b|},W^{r,p}_{a-|b|}).
\end{equation*}
This proves the continuity of $\Phi$ on $\cV_0$.
\end{proof}

We can now state the well-known result giving an iterative time-order
expansion of the operator $U(1)$. Let $L = L_0 + V$ as in Equation 
\eqref{eq.splitting} (that is, $L, L_0 \in 
\bL_\gamma$ with $L_0$ time independent).

\begin{proposition} \label{prop.perturbative}
Let $V(t) = L(t) - L_0$ be as in \eqref{eq.splitting}
and $U = U^L$ the evolution system generated by $L$. 
For any $d\in \ZZ_+$, we have
\begin{multline}\label{eq.perturbative}
  U(1) = \, e^{L_0} + \int_{\Sigma_1} e^{ \tau_0L_0}V( \tau_1)e^{
    \tau_1L_0}\, d\tau_1  + \dots \\
  + \int_{\Sigma_{d}}e^{ \tau_0L_0}V( \tau_1)e^{ \tau_1L_0} \dots
  e^{ \tau_{d-1}L_0}V(\tau_{d})e^{\tau_{d} L_0}\, \prod_{j=1}^{d-1} d\tau_j \\
  + \int_{\Sigma_{d+1}}e^{\tau_0L_0}V(\tau_1)e^{\tau_1L_0} \dots
  e^{\tau_{d}L_0}V(\tau_{d+1})U(\tau_{d+1})\, \prod_{j=1}^{d} d\tau_j,
\end{multline}
where each integral is a well-defined
Banach-valued Riemann-Stieltjes integral.
\end{proposition}

The positive integer $d$ will be called the {\em iteration level} of
the approximation. Later on, $V$ will be replaced by a Taylor
approximation of $L$, so that $V$ will have polynomial coefficients
in $x$ and $t$.

\begin{proof}
We proceed inductively on $d$. First, we note that each term in \eqref{eq.perturbative} 
is well defined by
Lemma \ref{lemma.perturbative}.

Formula \eqref{eq.perturbative} $d=1$ is just Equation \eqref{eq.duhamel} written in 
terms of operators.
Suppose now that the formula holds for $d-1$:
\begin{align}
  U(1) &= e^{L_0} + \int_{\Sigma_1} e^{(1-\s_1)L_0}V(\s_1)e^{\s_1L_0}\,
  d\s_1\nonumber\\
  &+ \int_{\Sigma_2} e^{(1-\s_1)L_0} V(\s_1) e^{(\s_1-\s_2)L_0} V(\s_2)
  e^{\s_2L_0} \,d\s_1 d\s_2 \nonumber\\
  &+ \dots + \int_{\Sigma_{d-1}}e^{(1-\s_1)L_0}V(\s_1) \dots
  e^{(\s_{d-2}-\s_{d-1})L_0}V(\s_{d-1})U(\s_{d-1})\, \prod_{j=1}^{d-1} d\s_j\nonumber.
\end{align}
Applying the formula for $d=1$ to $U(\s_{d-1})$ then gives:
\begin{multline} \label{expansion.formula}
  U(1) = e^{L_0} + \int_{\Sigma_1} e^{(1-\s_1)L_0}V(\s_1)e^{\s_1L_0}\,
  d\s_1 \\
%
%
 + \dots + \int_{\Sigma_{d-1}}e^{(1-\s_1)L_0}V(\s_1) \dots
  e^{(\s_{d-2}-\s_{d-1})L_0}V(\s_{d-1})U(\s_{d-1})\, \prod_{j=1}^{d-1} d\s_j \\
  = e^{L_0} + \int_{\Sigma_1} e^{(1-\s_1)L_0}V(\s_1)e^{\s_1L_0} \,d\s_1 \\
%
%
  +\int_{\Sigma_{d-1}}\int_{0}^{\s_{d-1}} e^{(1-\s_1)L_0}V(\s_1) \dots
  V(\s_{d-1})e^{(\s_{d-1} - \s_d) L_0}
  V(\s_d) U(\s_d)\, \prod_{j=1}^{d-1} d\s_j \, d\s_d.
%
%
\end{multline}
which is \eqref{eq.perturbative} for $d$.
\end{proof}

By sending $d\to +\infty$, we formally represent the evolution system
as a series of iterated, time-ordered integrals. Such series appear in
different contexts and are known as {\em Dyson series} in the Physics
literature.

\section{Dilations and Taylor expansion}
\label{sec5}

In this section we employ suitable space-time dilations to reduce
the computation of the Green's function $\cG_{t,t'}^L$ to that of a related
operator $L^s$ at fixed time $t=1$ where $s=\sqrt{t}$ .
For given, fixed  $s>0$, we then obtain an expression  of  the Green's
function associated to $L^s$ by Taylor expanding its coefficients as
functions of $s$ up to order $n$ and combining such expansion with the
time-ordered expansion \eqref{expansion.formula} up to level $d$.
In particular, the Taylor expansion will provide a natural choice for
the operator $L_0$ and $V(t)$ to which the splitting
\eqref{eq.splitting} of $L^s$ applies. We follow here closely
\cite{CCMN}, which treats the case of time-independent operators.
In particular, we use the crucial observation from that paper
that, for any second order differential operator with constant
coefficients $L_0$ and any differential operator with polynomial
coefficients $L_m$, we have $e^{L_0}L_m = \tilde{L}_m e^{L_0}L_m$ for some
other differential operator with polynomial
coefficients $\tilde{L}_m$. (We actually extend this result to 
higher order operators $L_0$.) Similar methods, including the time
dependent case, were employed in \cite{WenThesis, CCLMN, CMN-preprint,
Pascucci15, Pascucci17, Pascucci12}.

{\em Throughout this section, we fix an arbitrary dilation center
$z \in \bR^N$.}

\subsection{Parabolic rescaling}

For any sufficiently regular functions $v(t,x)$ and $f(x)$, we set
\begin{subequations} \label{dilation.function}
 \begin{equation} \label{dilation.function1}
   v^s(t,x):=v(s^2t,z+s(x-z)),
 \end{equation}
\begin{equation} \label{dilation.function2}
   f^s(x):=f(z+s(x-z)).
\end{equation}
\end{subequations}
We therefore interpret $s$ as the dilation factor and $(0,z)$ as the
dilation center.

For any given operator $L(t)\in \LL$, we similarly define
\begin{align} \label{dilation operator}
  L^s(t) &: = \sum_{ij=1}^N
  a_{ij}^s(s^2t,z+s(x-z)) \partial_i \partial_j + s
\sum_{i=1}^N
  b^s_i(s^2t,z+s(x-z)) \partial_i \nonumber \\
  &\qquad \qquad \qquad + s^2c^s(s^2t,z+s(x-z)).
\end{align}

It is not difficult to show that, if $u(t,x)$ is a solution of Equation \eqref{eq.def.IVP}, 
then $u^s(t,x)$ given by \eqref{dilation.function} is a solution of the
following IVP:
\begin{equation}\label{eq.def.IVP.dilated}
\begin{cases}
  \D_t u^s(t, x) - L^su^s(t, x) =0 & \hspace{1.0cm} \mbox{in}\;
  (0,\infty)\times \bR^N\\
  u^s(0, x) =g^s(x),   & \hspace{1.0cm}
  \mbox{on}\; \{0\}\times \bR^N\,.
\end{cases}
\end{equation}

Clearly, if $L = (L(t))_{t \in I} \in \LL_\gamma$, then $L^s$ is an operator in 
the same class, but with a possibly different $I$. Since our estimates will be uniform
up to a finite time, we shall assume from now on that $I = [0, T]$, for some fixed $T > 0$,
and we shall consider $L^s(t)$ only for $s \in [0, 1]$ and $t \in [0, T] = I$. Based
on our earlier discussion $L^s = (L^s(t))_{0 \le t \le T}$ generates an evolution system, 
which we denote
by $U^{L^s}$. The fundamental solution of the IVP
\eqref{eq.def.IVP.dilated} will be denoted instead with
 $\cG_t^{L^s}(x,y)$. The Green's functions for the original and dilated
 problems are simply related via a change of variables.

\begin{lemma}\label{lemma.rescaling} 
Given any $z\in \RR^N$ and $s>0$, we have
\begin{equation*}
  \cG_{t}^{L}(x,y) \seq s^{-N}\cG_{s^{-2}t}^{L^s} 
  \Big (z + \frac{x-z}{s}, z+\frac{y-z}{s} \Big ) \,.
\end{equation*}
In particular, when $s=\sqrt{t}$,
\begin{equation} \label{Green.relation}
  \cG_{t}^{L}(x,y) = t^{-N/2}\cG_1^{L^{\sqrt{t}}} \Big (z +
  \frac{x-z}{\sqrt{t}}, z + \frac{y-z}{\sqrt{t}} \Big ) \,.
\end{equation}
\end{lemma}

By this lemma, it suffices to approximate $\cG_1^{L^{s}}(x,y)$ and set $s=\sqrt{t}$.

\subsection{Taylor expansion  of the operator $L^s$}

We next Taylor expand the coefficients of the
operator  $L^s$, given by \eqref{dilation operator},
up to order $n\in \ZZ_+$,
as functions of $s>0$. The purpose of this Taylor expansion is to
replace the operator $V$ in \eqref{eq.splitting} with operators having
polynomial coefficients for which  the time-ordered
integrals appearing in \eqref{expansion.formula}  can be explicitly
computed as in \cite{CCMN}.

We obtain the representation
\begin{equation} \label{eq.Taylorexp}
  L^s = L_0+\sum_{m=1}^n s^m L_m + s^{n+1}L_{n+1}^{s,z} \,,
\end{equation}
where
\begin{equation}\label{eq.expansion}
  L_m = \left. \frac{1}{m!}\left( \frac{d^m}{ds^m} L^s \right ) \right
  |_{s=0}, \quad 0\leq m\leq n \,,
\end{equation}
and $L_{n+1}^s$ comes from the remainder of the Taylor expansion.
For $m$, $L_m = (L_m(t))_{0 \le t \le T}$ is a family of differential
operators indexed by $t \in [0, T]$ with coefficients that are polynomials
in $(x - z)$, but are independent of $s$. Globally, the family $L_m$ depends
polynomially on $t$. That is, for $m \le n$,
\begin{equation}
   L_m(t) \seq \sum_{ijk\alpha } a_{ijk\alpha}^{[m]} (x-z)^\alpha t^k \pa_i \pa_j +
   \sum_{ik\alpha } b_{ik\alpha}^{[m]}(x-z)^\alpha t^k \pa_i + \sum_{k\alpha } 
   c_{k\alpha}^{[m]}(x-z)^\alpha t^k\,,
\end{equation}
(finite sums) with the coefficients $a^{[m]}, b^{[m]}, c^{[m]} \in \RR$
obtained from the partial derivatives of the coefficients of $L$ at $(t, x) = (0, z)$. 
However, $L_{n+1}^s$ does depend on $s$. 

In what follows, we obtain a perturbative expansion of the form
\eqref{expansion.formula} for $U^{L^s}(1)$ with $V_j$
 replaced by the operator $L_j$
introduced above. In justifying such an expansion, we will need to apply
Lemma \ref{lemma.perturbative}, reduced to a special case. We record this
special case in the following corollary for future use.
We notice that $L_0(t)$ is independent of $t$, so we shall
write simply $L_0$. Let $\sigma_j := \tau_j + \tau_{j+1} + \ldots + \tau_{k}$, if
$\tau \in \Sigma_k$, $\sigma_k = \tau_k$, as before, see Remark \ref{rem.simplex}.

\begin{corollary}\label{cor.wel}
Let $L(t) \in \LL_\gamma$, let $k\in \ZZ_+$, and let $L_{m}$, 
$0 \le m \le n+1$, be from the Taylor expansion of $L$, Equation 
\eqref{eq.Taylorexp}. For $\tau \in \Sigma_k$, let
us set
\begin{equation*}
    \Phi(\tau) \ede e^{\tau_0L_0} L_{j_1}(\sigma_1) e^{\tau_1L_0} 
    L_{j_2} (\sigma_2)  \, \dots \, L_{j_{k-1}}(\sigma_{k-1}) e^{\tau_{k-1}
      L_0} L_{j_k}(\sigma_k) E(\tau_k), 
\end{equation*}
with $0 \le j_i \le n+1$ and either 
$E(\tau_k)=e^{\tau_k L_0}$ or $E(\tau_k)=U^{L^s}(\tau_k)$. Then, for any
$b=(b_1,\ldots,b_k)\in \RR^k_+$, $a, r, s \in \RR$,
and $1<p<\infty$, $\Phi : \Sigma_k \to \cL(W^{s, p}_{a,z}(\RR^N), W^{r,p}_{a- |b|}(\RR^N))$
is continuous
\end{corollary}

\subsection{Asymptotic expansion of the evolution system} \label{s.Asymptotic}

In this section, we define an approximation $\cG_{t, s}^{\mu}$ of
the evolution system $U(t,s)$ satisfying the conditions of
Theorem \ref{thm.b.error}.

\begin{definition}[Spaces of Differentials]
Given non-negative integers $a,b$, we denote by $\cD(a,b)$ the vector
space of all differential operators of order at most $b$
with coefficients that are
polynomials in $x$ and $t$ of degree at most $a$ .
We extend this definition to negative indices by defining
$\cD(a,b) = \{ 0\}$ if either $a$ or $b$ is negative. By the degree of an
operator  $A$,
we mean the highest power of the polynomials appearing as coefficients
of  $A$.
\end{definition}

\begin{definition}[Adjoint Representation] 
\label{def.ad.recursive}
For any two operators $A_1 \in \cD(a_1,b_1)$ and $A_2 \in
\cD(a_2,b_2)$ we define $\adj_{A_1}(A_2)$ by
\begin{equation*}
  adj_{A_1}(A_2) \ede [A_1, A_2] \seq A_1 A_2 - A_2 A_1 \seq - [A_2, A_1]\,,
\end{equation*}  
and, for any integer $j \geq 1$, we define $\adj^j_{A_1}(A_2)$ recursively by
\begin{equation*}
  \adj^j_{A_1}(A_2) \ede \adj_{A_1}(\adj^{j-1}_{A_1}(A_2)) \,.
\end{equation*}
\end{definition}

Above, the iterated commutators are well defined if we take the space $C^\infty_c(\RR^N)$
as common domain $\cD$ of $A_1$ and $A_2$, for instance.

\begin{lemma} \label{lemma.order-of-commutators}
Suppose $A_1 \in \cD(a_1,b_1)$ and $A_2 \in \cD(a_2,b_2)$. Then for
any integer $k \geq 1$,
\begin{align*}
  {\rm ad}_{A_1}^k (A_2) \in \cD(k(a_1 -1) + a_2, k(b_1 -1) + b_2).
\end{align*}
\end{lemma}

\begin{proof}
We have that
$  \adj_{A_1}(A_2) \in \cD(a_1 -1 +a_2, b_1 - 1 + b_2).$
The result then follows by iterating $k$ times this relation.
\end{proof}


As in \cite{CCLMN, CCMN}, we obtain the following consequence of this lemma.

\begin{proposition}\label{prof.Hadamard} Let $Q \in \cD(0, n)$
and $Q_m \in \cD(m, m')$. We have the following:
\begin{enumerate}
   \item $\adj_{Q}^{m+1}(Q_m) = 0$; 

   \item Consequently, the following sum is finite
\begin{equation*}
   \exp( \adj_{Q}) (Q_m) \ede \sum_{j \ge 0} (j!)^{-1} \adj_{Q}^j(Q_m)\,.
\end{equation*}

\item $\exp( \adj_{Q}) (P_1P_2) = \exp( \adj_{Q}) (P_1)\exp( \adj_{Q}) (P_2)$
for all $P_1, P_2$ in the algebra  $\cD := \cup_{n,n'} \cD(n, n')$.

\item Assume that $Q$ generates a $c_0$-semigroup $e^{t Q}$ on $L^2(\RR)$, $t \ge 0$, 
then
\begin{equation*}
   e^{Q} Q_m = \exp( \adj_{Q}) (Q_m) e^{Q} \,.
\end{equation*}
\end{enumerate} 
\end{proposition}

\begin{proof}
The first relation follows from $\adj^k_{Q}(Q_m) \in \cD(m-k,m'+ k(n-1))$
and the fact that the later space is 0 when $k>m$. This then gives
immediately that $\exp( \adj_{Q})$ is defined. The third relation
follows from the fact that $\adj_{Q}$ is a derivation of $\cD$
and the exponential of a derivation (when defined) is an algebra 
isomorphism. Finally, to prove the last relation, let us consider the
function $F(t) := e^{tQ} Q_m - \exp( \adj_{Q}) (Q_m) e^{tQ}$. It
is a continuous function with values in $\cL(\rho_w^{-a} L^2(\bR^N), \rho_w^{a}L^2(\bR^N))$
for $a$ large ($a \ge m'+(m+1)(n-1)$). Then $F(0) = 0$ and $F'(t) = \adj_{Q}(F(t))$. Hence
$F(t) = 0$ for all $t > 0$. 
\end{proof}


Consequently, coming back to our problem, we obtain an automorphism 
$\phi_\theta : \cD \to \cD$ of the algebra $\cD := \cup_{n,n'} \cD(n, n')$,
given by the formula $\phi_{\theta}(Q) e^{\theta L_0} = e^{\theta L_0} Q$.
See also \cite{WenThesis, CMN-preprint, Pascucci15, Pascucci17, Pascucci12}.

\begin{lemma}\label{lemma.commute}
Let $m$ be a fixed positive integer and $L_m$, $0 \le m \le n$, be defined as in
\eqref{eq.expansion}, then for any $\theta \in \RR$,
\begin{equation} \label{eq.commute}
  e^{(1-\theta) L_0}L_m(\theta) \seq P_m(\theta,x-z,\partial)e^{(1-\theta) L_0},
\end{equation}
where $P_m(\theta,x-z,\partial) := \phi_{1-\theta} (L_m(\theta))$
a differential operator with coefficients polynomials in $\theta$ and $(x-z)$.
(There is no $t$, since we specialized at $t = \theta$ in the formula for $L_m$.)
\end{lemma}

Next, we rewrite equation \eqref{expansion.formula} in a more
computable and explicit form.
We recall that $d$ is the level of the iteration in the Dyson series
and $n$ is the order of the Taylor expansion of $L^s$.
In principle, $d$ and $n$ are unrelated, but we will find it
convenient later on to choose $d=n$.

For ease of notation, we write $L_{n+1}^s=L_{n+1}$, even though this operator 
does depend on $s$. Inserting \eqref{eq.Taylorexp} into  \eqref{expansion.formula}  and
collecting iterated integrals in the same number of variables, we
have:
\begin{align}
  & U^{L^s}(1) =
  e^{L_0} + \sum_{k=1}^d \,
    \sum_{ \substack{i=1,\ldots,k\\ 1\leq \a_i\leq n+1}} s^{\a_1+\dots
    +\a_k} \int_{\Sigma_{k}} e^{ (1-\sigma_1)L_0}\, L_{\a_1}
    (\sigma_1) \, e^{(\sigma_1-\sigma_2)L_0}  \nonumber\\
  & \dots\, e^{(\sigma_{k-1}-\sigma_k)L_0}\, L_{\a_{k}}(\sigma_{k})\,e^{\sigma_{k}
    L_0}\, d\sigma
  + \sum_{ \substack{i=1,\ldots, d+1\\ 1\leq \a_i\leq n+1}} s^{\a_1+\cdots
  +\a_{d+1}} \int_{\Sigma_{d+1}} e^{ (1-\sigma_1)L_0}\,  \nonumber\\
   & \qquad \qquad \cdot\, L_{\a_1}(
  \sigma_1) \, e^{(\sigma_1-\sigma_2)L_0} \dots e^{(\sigma_{d}-\sigma_{d+1})L_0}
  \, L_{\a_{d+1}}(\sigma_{d+1})\, U(\sigma_{d+1})\,d\sigma,
 \label{expand}
\end{align}

To simplify the above expression, we  now introduce some helpful
combinatorial notation to keep track of the indexes

\begin{definition}\label{def.index.set}
For any integers $1 \le k \le d+1$ and
$1\leq \ell\leq (n+1)(d+1)$,
we  denote by
\ $\fA_{k,\ell}$ the set of multi-indexes $\alpha = (\alpha_1, \alpha_2,
\ldots, \alpha_k) \in  \{0, 1, \ldots , n+1 \}^{k}$, 
such that $|\alpha| = \sum \alpha_j =
\ell$.  Furthermore, we let  \ $\fA_{0} =
\{\emptyset\}$.
\end{definition}

Clearly, since $\alpha_i\geq 1$, the set $\fA_{k,\ell}$ is empty if
$\ell<k$. If $\alpha \in \fA_{k,\ell}$, then $\ell$ represent the order in powers
of $s$ of the corresponding term in \eqref{expand}, while $k$
represents the level of iteration in the time-ordered expansion.

For each $\a \in \fA_{k,\ell}$, we then set
\begin{equation*}
\Lambda_{\a}=\int_{\Sigma_{k}} e^{ (1-\sigma_1)L_0} L_{\a_1}(
\sigma_1)\,e^{(\sigma_1-\sigma_2)L_0} \dots e^{
(\sigma_{k-1}-\sigma_k)L_0} L_{\a_{k}}(\sigma_{k})\, e^{\sigma_{k}
L_0}\,d\sigma,
\end{equation*}
if $k<d+1$, and
\begin{equation*}
  \Lambda_{\a} = \int_{\Sigma_{d+1}}e^{ (1-\sigma_1)L_0} L_{\a_1}(
  \sigma_1)\,e^{(\sigma_1-\sigma_2)L_0} \dots e^{
    (\sigma_{d}-\sigma_{d+1})L_0}L_{\a_{d+1}}(\sigma_{d+1})
 \, U^{L^s}(\sigma_{d+1})\,d\sigma,
\end{equation*}
if $k=d+1$, respectively.

A simple but useful result about $\Lambda_{\a,z}$ is the following lemma,
which we record for later use.

\begin{lemma}\label{lemma.commutator.explicit}
Recall the polynomials $P_k$ of Lemma \ref{lemma.commute}. 
For any given multi-index $\a\in \fA_{k,\ell}$ with $k \leq d$
and $1\leq \alpha_ i\leq n$, $i=1,\ldots,k$,
\begin{equation*}
  \Lambda_{\a} \seq \mathcal{P}_{\a} (x-z,\partial)e^{L_0}
\end{equation*}
where 
\begin{align*}
  \mathcal{P}_\a(y,\partial)&=\int_{\Sigma_{k}}P_{\a_1}(\sigma_1,y,\partial) 
  P_{\a_2}(\sigma_2,y,\partial)\cdots
  P_{\a_k}(\sigma_{k},y,\partial)d\sigma 
\end{align*}
is a differential operator with coefficients polynomials in $y$ (in 
particular, it is independent of $t$ or $s$).
\end{lemma}

\begin{proof}
Applying Lemma \ref{lemma.commute} repeatedly gives
\begin{align*}
  \Lambda_{\a,z}= &\int_{\Sigma_{k}}e^{ (1-\sigma_1)L_0}L_{\a_1}(
  \sigma_1)e^{(\sigma_1-\sigma_2)L_0} \dots e^{
    (\sigma_{k-1}-\sigma_k)L_0}L_{\a_{k}}(\sigma_{k})e^{\sigma_{k}
    L_0}d\sigma\\
  =&\int_{\Sigma_{k}}P_{\a_1}(\sigma_1,x-z,\partial)e^{(1-\sigma_2)
    L_0}\cdots e^{
    (\sigma_{k-1}-\sigma_k)L_0}L_{\a_{k}}(\sigma_{k})e^{\sigma_{k}
    L_0}d\sigma\\%
  & \qquad \qquad \vdots\\
  =  \Big ( \int_{\Sigma_{k}}  &  P_{\a_1}(\sigma_1,x-z,\partial)
  P_{\a_2}(\sigma_2,x-z,\partial) \cdots
  P_{\a_k}(\sigma_{k},x-z,\partial)d\sigma \Big ) e^{L_0}.
\end{align*}
This completes the proof.
\end{proof}

To further simplify some of the formulas, we define
\begin{equation}\label{def.lambda_z_l}
  \Lambda^\ell =\sum_{k=1}^{\min(\ell,d+1)} 
  \sum_{\a\in \fA_{k,\ell}} \Lambda_{\a}
\end{equation}
For convenience, we let \ $\Lambda^0 = e^{L_0}$.

We combine the results obtained so far in this section in the
following representation theorem. We will  perform an error analysis
in the Sobolev spaces $W^{k,p}_{z,a}$ in Section \ref{sec6}.

\begin{lemma}[Definition of the local approximation]
\label{lemma.def.local}
Let $d$ be the iteration level in the
time-ordered expansion \eqref{expansion.formula}, let $n$ be the order of
the Taylor expansion  \eqref{eq.Taylorexp} of $L^s$, as before, and
let $m \in \ZZ_+$. Let 
\begin{equation*}
    E^{s}_{m,d,n} = \sum_{\ell=m+1}^\infty s^{\ell-m-1}\Lambda ^{\ell}\,.
\end{equation*}
(The sum is actually finite.) Then
\begin{equation*}  
   U^{L^s}(1, 0) \, = \  e^{L_0} +\sum_{\ell=1}^m s^\ell \Lambda ^\ell
   +s^{m+1}E^{s}_{m, d,n}\,.
\end{equation*}
Assume that $m \le \min \{d, n \}$. Then 
$\Lambda ^\ell$ does not depend on $d$, $n$, or $s$, and, consequently, 
$E^{s}_{m,d,n}$ also does not depend on $d$ and $n$.
\end{lemma}

\begin{proof}
This follows from the fact that, if $\alpha \in \mfkA_k^{\ell}$,
then $k \le \ell := \alpha_1 + \alpha_2 + \ldots + \alpha_k$, 
since all $\alpha_i \ge 1$.
\end{proof}

Consequently, when $m \le \min \{d, n\}$, we shall write $E^{s,z}_{m} = 
E^{s,z}_{m,d,n}$, since $E^{s,z}_{d,n}$ does not depend on $d$ and $n$.

\begin{remark}\label{rem.Gapprox}
The idea pursued here (following \cite{CCMN}) relies on the 
following three analysis points
\begin{itemize}
  \item $U^{L^s}(t, t')$ depends smoothly on $s \in [0, 1]$;
  \item we can explicitly identify  $U^{L^0}(t, t') = e^{(t-t')L_0}$;
  \item the sum $e^{L_0} +\sum_{\ell=1}^m s^\ell \Lambda ^\ell$ is 
  the Taylor polynomial
of $U^{L^s}(1, 0)$ at $s = 0$.
\end{itemize}
Note that $L_0$ is obtained from the operator $L$ by freezing its
coefficients at $(0, z)$ ($t = 0$ in time and $z$ in space). We can thus 
try to approximate 
$U^{L^s}(1, 0)$ with its Taylor polynomial. In turn, after rescaling back, this approximation
will yield an approximation of $U^{L}(s^2, 0)$, that is, for short time. Note that
$U^{L^s}(1, 0)$ does not exhibit any singularities at $s = 0$, but rescaling back 
introduces a strong singularity at $s =0$ in $U^{L}(s^2, 0)$, however, repeating
ourselves, that singularity is entirely due to the rescaling. The next section will
make this construction explicit to define the approximate Green function of
$U^{L}(t, s)$ for $t-s>0$ small. 
\end{remark}

\section{The approximate Green function and error analysis}
\label{sec6}

In this section we introduce our approximate Green function, 
we prove Theorem \ref{thm.maintheorem}, and 
we complete our error analysis. Our error estimates are
using the norm of linear maps between weighted Sobolev
spaces. A different kind of estimate (pointwise in $(x,y)$)
was obtained in \cite{Pascucci12}.

\subsection{Definition of the approximate Green function}
We are now ready to introduce our approximation of the Green function 
\begin{equation*}
   \cG_{t, s}^{L} (x,y)  \ede U^L(t, s)(x, y)
\end{equation*}
of the operator $U^L(t, s)$ following the idea outlined in Remark
\ref{rem.Gapprox}. Since the problem is translation invariant, 
we may assume $s = 0$ and thus we shall write $\cG_{t, 0}^{L} (x,y) = 
\cG_{t}^{L} (x,y)$. Soon, we will replace $z$ (which was fixed in the
previous section) with a function of $x$ and $y$. We first introduce
the conditions that such a function must satisfy.

\begin{definition} \label{def:admissible}
A smooth function $z : \RR^{2N} \to \RR^N$ will be called
{\em admissible} if $z(x,x) = x$, for all $x \in \RR^N$ and all 
partial derivatives (of all positive orders) of $z$ are bounded.
\end{definition}

A typical example is $z(x,y) = \lambda x + (1-\lambda)y$, for some
fixed parameter $\lambda$.  A simple application of the mean value
theorem gives that $\<z-x\> \le C \<y-x\>$ for some constant $C>0$.
{F}rom the point of view of application, $z(x,y)=x$ will give us the
simplest formula to approximate the Green function. However, as discussed 
in \cite{CCLMN}, other more suitable choices are possible, for 
instance, $z(x, y) = (x+y)/2$ seems to be better. In what follows,
we fix an admissible $z=z(x,y)$. We now fix for the rest of the
paper an admissible function $z : \RR^{2N} \to \RR^N$. It will be
the dilation center used to approximate the Green functions at $(x,y)$.
 
Assume we want an approximation of order $m$ (that is, up to
$s^m = t^{m/2}$). We shall use the formulas and the results of
Lemma \ref{lemma.def.local}. We shall choose then in that Lemma
$n, d \ge m$, so that the terms $\Lambda ^\ell$ are independent
of $s$ (and $t$) and $E^{s}_{m,d,n}$ is independent of $d$ and $n$,
so we can write $E^{s}_{m,d,n} = E^{s}_{m}$ for the ``error term.''
Motivated by Lemmata \ref{lemma.rescaling} and \ref{lemma.def.local}, 
we now introduce the following. 

\begin{definition}\label{def.orderm}
We assume $m \le \min \{d, n \}$ and
let the {\em order $m$ approximation} $\cG_t^{[m]}(x,y)$ of the 
Green function $\cG_{t}^{L} (x,y)  \ede \cG_{t, 0}^{L} (x,y)$ of 
$U^L(t, 0)$ be
\begin{equation*}
  \cG_t^{[m]}(x,y) \ede \sum_{\ell=0}^m \, t^{(\ell-N)/2}
  \Lambda ^\ell \Big (z+ \frac{x-z}{\sqrt{t}}, z+ \frac{y-z}{\sqrt{t}} \Big )\,.
\end{equation*}
\end{definition}

For this definition, it suffices to
choose $n = d = m$, but for the proof of our error estimates, the 
freedom to choose much larger $n$ and $d$ will be useful. 
This will be especially the case when dealing with the error term:
\begin{align}\label{eq.def.error}
   \widetilde{E}^{t}_{m}(x,y) \ede & t^{-(m+1)/2}\, \Big [\, \cG_t^{L}(x,y) - 
   \cG_t^{[m]}(x,y) \, \Big]\\
    \seq & E^{t}_{m} \Big (z+ \frac{x-z}{\sqrt{t}}, z+ \frac{y-z}{\sqrt{t}} \Big )\,.,
\end{align}

By replacing $L$ with a translation of size $t'$ in time, we define 
similarly the approximation $\cG_{t, t'}^{[m]}(x,y)$ using
the $[m   ]$--approximate kernel at $(t-t', 0)$ for this translated 
operator.

\subsection{Convergence Analysis} \label{ssec.error.analysis}

In this section, we show that our approximate Green function
$\cG_{t, t'}^{[m]}(x,y)$ satisfies the assumptions of Theorem 
\ref{thm.b.error}. We shall primarily
use pseudo-differential techniques. For all relevant
properties of pseudo-differential operators, we refer to \cite{Tay, Treves}.
For the moment, we continue to keep $z$, the dilation center, fixed.

We start by analyzing in more detail the properties of
the operators $L_m$ in expansion \eqref{eq.expansion}. We recall that \ 
$\langle x\rangle_z= \langle x-z\rangle$.
We also recall that  $L_m$, $0\leq m \leq n+1$, are second-order differential 
operator with polynomial coefficients, independent of the dilation factor $s$.
Moreover, $L_m$ has coefficients of degree at most $m$ in $x-z$. An
immediate consequence of this fact is recorded in the following lemma.

\begin{lemma}\label{lemma.poly.growth}
The family
\begin{equation*}
    \{\<x\>_{z}^{-j} L_j^{z},  \ \<x\>_{z}^{-n-1}
    L^{s,z}_{n+1}; \ s \in (0, 1],\ z \in \RR^N, \ j=0, \ldots, n+1\}
\end{equation*}
defines a bounded subset of $\bL$.
\end{lemma}

For convenience, we recall we denote $L^{s,z}_{n+1}$ by
$L_{n+1}^z$, even though this operator depends on $s$.
The next Lemma allows to change the center of the dilation. This change is needed 
when $z$ is replaced by a function $z=z(x,y)$.  It also allows to reduce to the 
case $a=0$ to establish bounds in $W^{k,s}_{a,z}$, as long as $a$ belongs to a 
bounded set.

\begin{lemma}\label{cor.exp.growth}
For each given $\epsilon > 0$, the family
\begin{equation*}
   \{e^{-\epsilon \langle z\rangle_w} e^{-\epsilon \<x\>_{w}} L_j^{z}, \
s \in (0, 1],\ z, w \in \RR^N, \ j=0, \ldots, n+1\}
\end{equation*}
is a bounded subset of $\bL$.
\end{lemma}

\begin{proof} 
A simple calculation shows that
\begin{equation*}
     \<x-z\> - \<x-w\> \leq \<w-z\>.
\end{equation*}
Therefore $e^{\epsilon\,(\<x-z\>-\<x-w\>-\<w-z\>)}\leq 1$, and hence the family
\begin{equation*}
  e^{\epsilon(\<x-z\>-\<x-w\>-\<w-z\>)}\, e^{-\epsilon \<x\>_z}\,
  L_j^z=e^{-\epsilon \<z-w\>}\, e^{-\epsilon \<x\>_{w}}\, L_j^z
\end{equation*}
is bounded for $s\in (0,1]$ and $j=0,1,2,\cdots,n+1$ as claimed.
\end{proof}

Lemma \ref{lemma.perturbative} and Lemma \ref{cor.exp.growth}
yields the following result.

\begin{corollary}\label{cor.wel.def}
For any $\a_1,\a_2,\cdots,\a_k$ with $\sum_{i=1}^k\a_i=\ell$, the
operators
\begin{equation*}
  \Lambda_{\a,\ell}=\int_{\Sigma_{k}}e^{ \tau_0 L_0}L_{\a_1}(
  \tau_1) \, e^{\tau_1 L_0} \cdots e^{
    \tau_{k-1}L_0}L_{\a_{k}}(\tau_{k})e^{\tau_{k} L_0}d\tau,\quad
  k\leq d
\end{equation*}
and
\begin{equation*}
  \Lambda_{\a,\ell}=\int_{\Sigma_{d+1}}e^{ \tau_0 L_0}L_{\a_1}(
  \tau_1) \, e^{\tau_1 L_0} \cdots e^{
    \tau_{d}L_0}L_{\a_{d+1}}(\tau_{d+1})U(\tau_{d+1})d\tau
\end{equation*}
are bounded linear operators from $W_{a, z}^{s,p}$ to
$W_{a-\epsilon}^{r,p}$ for any $z \in \RR^N$, $r, s \in \RR$,
$1<p<\infty$, and $\epsilon > 0$. Moreover, we have that
\begin{equation*}
  \|\Lambda_{\a,\ell}\|_{W^{s, p}_{a, z}  ,  W^{r, p}_{a-\epsilon,w}}
  \le C_{s, r, p, a, \epsilon}\, e^{k\epsilon<z-w>},
\end{equation*}
for a bound $C_{s, r, p, a, \epsilon}$ that does not depend on
$z$. In particular, each $\Lambda_{\a,\ell}$ is an operator with
smooth kernel $\Lambda_{\a,\ell}(x, y)$.
\end{corollary}

In order to treat the resulting kernels and the resulting remainder term, 
Corollary \ref{cor.wel.def} is not sufficient and we need refined estimates. 
We address first the terms comprising $\cG^{[m]}_t$ of the expansion
introduced in Definition \ref{def.orderm}
via pseudo-differential calculus and treat the terms in the 
remainder next via direct kernel estimates.

\subsection{Bounds on $\cG_t^{[m]}$} \label{s.GtBounds}
We bound each operator $\Lambda_{\alpha}$ appearing
in Definition \ref{def.orderm} separately, where 
$\Lambda_{\alpha}$ is defined in \eqref{def.lambda_z_l}. To this end, we define the
operator
\begin{equation}\label{eq.error4}
  \cL_{s,\alpha} f(x) = s^{-N} \int_{\RR^N} \Lambda_{\alpha   }(z +
  s^{-1}(x - z), z + s^{-1}(y-z))f(y) dt \, ,
\end{equation}

We show below that, for an admissible function $z$, and $\alpha
= (\alpha_1, \ldots, \alpha_k) \in \fA_{k,\ell}$, $k \le n$,
$\alpha_i \le n$, the operator $\cL_{s,\alpha}$ is a
pseudo-differential operator with a good symbol. We shall
then use symbol calculus to derive the desired operator estimates.
By Lemma \ref{cor.exp.growth}, it is enough to assume $a=0$ in 
$W^{s,p}_{a,w}$.

A direct computation gives the following lemma, using the explicit form of 
the kernel of $e^{L_0}$ (which is known since $L_0$ is constant
coefficient).

\begin{lemma} \label{lemma.Lz}
Fix $z \in \RR^N$. Consider the operator $T
= (x-z)^{\beta}\partial_x^{\gamma} e^{L_0^z}$, where $\beta$ and
$\gamma$ are multi-indices and $a \in \CIb(\RR^N)$. Then the
distributional kernel of $T$ is given by
\begin{equation*}
    T(x,y) \seq  (x-z)^\beta (\pa_x^\gamma G)(z; x-y) \, .
\end{equation*}
\end{lemma}

The next theorem characterizes the symbol of $\cL_{s,\alpha}$ belonging to 
the principal term of the expansion

\begin{theorem}  \label{thm.mainbound}
Let $\alpha \in \fA_{k,\ell}$, $k \le n$, $\alpha_i \le n$.
Let $z: \RR^N \times \RR^N$ be an admissible function
Then there exists a uniformly bounded family $\{\varrho_{s}\}_{s\in (0,1]}$ in
$S^{-\infty} (\bR^N \times \bR^N)$ such that
\begin{equation*}
  \cL_{s,\alpha} \seq \sigma_s(x, D) \ede \varrho_{s}(x, sD), \quad
  \sigma_s(x, \xi) = \varrho_s(x, s\xi).
\end{equation*}
\end{theorem}

\begin{proof}
By Lemma \ref{lemma.commutator.explicit},  $\Lambda_{\alpha,
z}$ is a finite sum of terms of the form $ (x-z)^\beta
\pa_x^\gamma e^{L_0^z}$. We recall that $a$ is smooth with bounded 
derivatives of all orders. Let $k_z(x, y)$ be
the distribution kernel of $a(z)(x-z)^\beta \pa_x^\gamma e^{L_0^z}$
and set
\begin{equation*}
  K_s(x, y) := s^{-N}k_z(z + s^{-1}(x - z), z + s^{-1}(y-z)),\quad z = z(x, y).
\end{equation*}
By abuse of notation, we shall denote also by $K_s$ the integral
operator with kernel $K_s$. It is enough to show that there exists a uniformly
bounded family $\{\varrho_{s}\}_{s\in (0,1]}$ in $S^{-\infty}$ such
that
\begin{equation*}
  K_s \seq \varrho_s(x, sD).
\end{equation*}
A direct calculation shows that
\begin{multline*}
  K_s(x, y) =
  a(z) s^{-|\beta|-N} (x - z)^{\beta} \zeta(z,
  s^{-1}(x-y)), \quad z = z(x,y),
\end{multline*}
with $\zeta(z, x)$ the kernel of $\pa_x^\gamma e^{L_0^z}$.
Then the symbol of $K_s$, $\sigma_s(x, \xi)$ is given by
\begin{equation*}
  \sigma_s(x, \xi) = \int_{\RR^N} e^{-\imath y \cdot \xi} a(z)
   s^{-|\beta|-N} (x - z)^{\beta} \zeta(z, s^{-1}y) dy, \quad z = z(x,x-y).
\end{equation*}
If we denote
\begin{equation*}
  \varrho_s(x, \xi) = \int_{\RR^N} e^{-\imath y \cdot \xi} a(z)
   s^{-|\beta|} (x - z)^{\beta} \zeta(z, y) dy, \quad z = z(x,x-sy),
\end{equation*}
we have  $\sigma_s(x, \xi) = \varrho_s(x, s\xi)$. We show next
that $\varrho_s$ is a bounded family in $S^{-\infty}$. This follows 
from the continuity of multiplication with smoothing symbols, given 
that  that $a(z) \in S^0_{(1, 0)}$ and $s^{-1}(x_j - z_j(x, x-sy))
\in S^0_{(1, 0)}$ and they form bounded families for $s \in [0, 1]$.
\end{proof}

A simple change of variables and the definition of the symbol class 
$S^m_{1,0}$ gives the lemma below.

\begin{lemma}
Let $\varrho(x,\xi)$ be a symbol in $S^{-\infty}$, then $s^k
\varrho(x,s\xi)$ is a symbol in $S^{-k}_{1,0}$ uniformly bounded in
$(0,1]$ with respect to $s$.
\end{lemma}

The symbol calculus gives mapping properties on Sobolev spaces by standard results.

\begin{theorem}\label{thm.refined}
For any $1<p<\infty$, any $r\in \RR$,
\begin{equation}
   s^{k} \|\cL_{s, \alpha}\|_{W^{r,p} , W^{r+k,p}} \le C_{k, r, p},
\end{equation}
for a constant $C_{k, r, p}$ independent of $s$.
The same estimate is valid for the operator with kernel
$\cG_{t}^{[m   ]}(x, y)$.
\end{theorem}

By Definition \ref{def.orderm}, the above theorem translates into a 
corresponding bound on the principal part $\cG_{t}^{[m]}$
of the asymptotic expansion for the Green's function.

\begin{corollary} \label{cor.maintermbound}
Let $T > 0$ be fixed. For each $1<p<\infty$, $r\in \RR$, and 
any $f\in W^{r,p}$, the operator
$\cG^{[m]}_t$ with kernel $\cG^{[m]}_t(x,y)$ (that is, 
$\cG^{[m]}_tf(x) \ede \int_{\RR^N} \cG^{[m]}_t(x,y) f(y)\, dy$)
is uniformly bounded in $W^{r,p}$ for $t \in (0, T]$.
\end{corollary}

\subsection{Bounds on $\widetilde{E}^{t}_{m}$} \label{s.ErBounds}
In this subsection, we study the error
term $\widetilde{E}^{t}_{m}$ in \eqref{eq.def.error}.
To this end, we recall that, if $d$ and $n$ are large enough, 
both  $\cG_t^{[m]}(x,y)$ and  $\widetilde{E}^{t   }_{m}$
are  independent of $d$ and $n$. Next, we replace $m$ with  
$M\geq m+r-1$  in Definition \ref{def.orderm}, with $r>0$ 
to be chosen. Then we increase $d$ and $n$ accordingly to satisfy $d, n \ge M$, remembering that
$ \widetilde{E}^{t   }_{m}(x,y)$ does not depend on $d$ and $n$ as long as
they are $d, n \ge m$.
we can decompose $ \widetilde{E}^{t   }_{m}(x,y)$ as follows:
\begin{align}
   \widetilde{E}^{t   }_{m}(x,y) = \sum_{\ell=m+1}^M t^{(\ell-N-m-1)/2}
  \Lambda ^\ell(z+t^{-1/2}(x-z),z+t^{-1/2}(y-z)) \nonumber \\
  \qquad \qquad + t^{(M-m-N)/2}\, \widetilde{E}^{t   }_M(x,y). 
  \label{eq.ErrorTermSpit}
\end{align}
The first $M-m-1$ terms in this expressions are pseudo-differential operators of the 
type discussed in Subsection \ref{s.GtBounds}. The last term contains operators
$\Lambda_{\alpha,M}$ with either $\alpha \in \fA_{n+1, M}$ or for
some $\alpha_i=n+1$. In this range, we generally do not know whether
$\Lambda_{\alpha,l}$ is a pseudo-differential operator or not.
Instead of symbol calculus, it will be enough to apply a well-known result, 
sometimes referred to as Riesz's Lemma, which we recall for the reader's sake 
(see for example \cite[Proposition 5.1, page 573]{TayPDEI}).

\begin{lemma}
Assume $K$ is an integral operator with kernel $k(x,y)$
on a measure space $(X,\mu)$. If for all y and for all x, respectively,
\begin{equation}
  \int_X |k(x,y)|d\mu(x)\leq C_1, \int_X |k(x,y)|d\mu(y)\leq C_2
\end{equation}
 then $K$ is a bounded operator
on $L^p(X,\mu)$, $p\in [1,\infty]$. Moreover,
\begin{equation*}
  \|K\|\leq C_1^{1/p}C_2^{1/q}, \qquad 1/p+1/q=1.
\end{equation*}
\end{lemma}

Again, by Lemma \ref{lemma.same.class}, we need only consider 
the case $a=0$ in $W^{s,p}_{a.w}$.

\begin{lemma}\label{lem.rough}
Let $z : \RR^N \times \RR^N$ be admissible and let $1<p<\infty$. Then, for any 
$\alpha$ and any  $r\ge 0$, there exists a constant $C_{r,p,\alpha}>0$ such that
\begin{equation}
    s^{r} \|\cL_{s, \alpha}\|_{L^{p}, W^{r,p}} \le C_{r, p,\alpha}.
\end{equation}
\end{lemma}

\begin{proof}
By Riesz's Lemma it suffices to show that, for any multi-index $\gamma$ with $|\gamma|\leq k$,
\begin{equation}\label{reduced Riestz}
  \int_{\RR^N} s^{|\gamma|} |\partial_x^\gamma
  \cL_{s,\alpha}(x,y)|dy\leq C_1, \int_{\RR^N} s^{|\gamma|}
  |\partial_x^\gamma \cL_{s,\alpha}(x,y)|dx\leq C_2,
\end{equation}
where the constants $C_1$ and $C_2$ are independent of $x$ and
$y$ respectively.
We observe that $\partial_x^\gamma \cL_{s,\alpha}(x,y)$ is the sum of terms of the
form
\begin{equation}
  s^{-N-j}\partial_x^{\beta} \partial_z^{\beta'}
  \partial_y^{\beta^{''}}
  \Lambda_{\alpha   }(z+s^{-1}(x-z),z+s^{-1}(y-z))\cdot \xi(z),
\end{equation}
where $j\leq |\gamma|$ and $\xi(z)$ is the product of derivatives of
$z$ with respect to $x$, which is bounded as $z$ is admissible. Keeping $x, \, y$ fixed, 
we bound each of these terms, using the Schwartz Kernel Theorem, since $\Lambda_{\alpha   }$ 
is a smoothing operator:
\begin{multline}\label{estimate}
   |\partial_x^{\beta} \partial_z^{\beta'} \partial_y^{\beta^{''}}
  \Lambda_{\alpha   }(x,y)| = |\<\partial^\beta \delta_x,
  \partial_z^{\beta'} \Lambda_{\alpha   } \partial^{\beta^{''}}
  \delta_y\>|\\
   \leq C \|\partial^\beta \delta_x\|_{H^{-q}_{-\epsilon}}
  \|\partial_z^{\beta'} \Lambda_{\alpha   }\|_{H^{-q} \rightarrow
    H^q_{-\epsilon}}\| \partial^{\beta^{''}}\delta_y\|_{H^{-q}},
\end{multline}
where $\jap{\cdot}$ denotes again the pairing between smooth functions and compactly 
supported distributions. Above, we employed Lemma \ref{cor.wel.def} with $p=2$, $a=0$, 
and $w=z$. Next we estimate the three norms at the right hand side of the above 
inequality. Choosing $q>N+|\beta|$ gives for  all $\epsilon>0$,
\begin{equation*}\|\partial^\beta
  \delta_x\|_{H^{-q}_{-\epsilon}} :=
  \|e^{-\epsilon<x-z(x,y)>}\partial^\beta \delta_x\|_{H^{-q}} \leq
  Ce^{-\epsilon<x-z(x,y)>}
\end{equation*}
and similarly for $\partial^{\beta^{''}} \delta_y$.
Since all the coefficients and their derivatives of $L(t)$ are
bounded, $\partial_z^{\beta'} \Lambda_{\alpha   }$ satisfies the
same mapping properties as $\Lambda_{\alpha   }$. Thus by Corollary
\eqref{cor.wel.def},
  $\|\partial_z^{\beta'}
  \Lambda_{\alpha   }\|_{H^{-q} \rightarrow H^q_{-\epsilon}}
  \leq Ce^{\epsilon<z-x>}$.
Consequently,
\begin{equation*}
  |\partial_x^{\beta} \partial_z^{\beta'} \partial_y^{\beta^{''}}
  \Lambda_{\alpha   }(x,y)| \leq C e^{\epsilon<z-x>
  -\epsilon<x-z>} \leq C,
\end{equation*}
 and we obtain
\begin{multline*}
  |s^{-N-j} \partial_x^{\beta} \partial_z^{\beta'}
  \partial_y^{\beta^{''}}
  \Lambda_{\alpha   }(z+s^{-1}(x-z),z+s^{-1}(y-z)) \cdot \xi(z)|
%
%
  \leq Cs^{-N-|\gamma|}.
\end{multline*}
Finally, by the change of variable
$\lambda=\frac{y-x}{s}$, we verify that \eqref{reduced Riestz} holds.
The proof is complete.
\end{proof}

Lemma \ref{lem.rough} implies immediately

\begin{corollary} \label{thm.rough}
Let $z$ be admissible, and let $k\in \ZZ_+$, $1<p<\infty$. Then, for 
any $r\geq 0$ and $\alpha$, there exists a constant $C_{k,r,p,\alpha}>0$ 
such that
\begin{equation}
  s^{k+r} \|\cL_{s, \alpha}\|_{W^{r,p}, W^{r+k,p}} \le C_{k, r, p,\alpha}.
\end{equation}
\end{corollary}

Let $\mathcal{E}^{[m   ]}_t$ be the integral operator with kernel $\widetilde{E}^{t   }_m$.

\begin{theorem} \label{thm.cor.refined}
Under the hypotheses of Theorem \ref{thm.refined},  $\mathcal{E}^{[m   ]}_t$
satisfies
\begin{equation}
\|\mathcal{E}^{[m   ]}_t\|_{W^{r,p},  W^{r+k,p}} \leq C_{r,k,p,m}\,
s^{-k}.
\end{equation}
\end{theorem}

\begin{proof}
Recall the splitting \eqref{eq.ErrorTermSpit}.
Then, applying Theorems \ref{thm.refined} and
\ref{thm.cor.refined} gives
\begin{multline*}
    \|\mathcal{E}^{[m   ]}_t\|_{W^{r,p}  ,  W^{r+k,p}} \le
    \sum_{\ell = m + 1}^{M} s^{\ell-m-1}
    \sum_{k = m+1}^{\ell} \sum_{\alpha \in \fA_{k, \ell}}
    \|\cL_{\alpha   }\|_{W^{r,p}  ,  W^{r+k,p}}\\  + s^{M+1-m}
    \|\mathcal{E}^{[M   ]}_t\|_{W^{r,p}  ,  W^{r+k,p}}
    \le Cs^{-k} (1  + s^{M+1-m}s^{-r}) \le Cs^{-k} .
\end{multline*}
This completes the proof.
\end{proof}

Our main result, Theorem \eqref{thm.maintheorem}, now follows from 
Definition \ref{def.orderm} the expansion 
the error analysis of this section.

\begin{remark}
It is not difficult to show that the approximation introduced in Theorem
\ref{thm.maintheorem} is invariant under affine transformations, a useful 
fact in applications. We refer to \cite{WenThesis} for more details.
\end{remark}

Combining Theorems \ref{thm.maintheorem} and \ref{thm.cor.refined}
with Theorem \ref{thm.b.error} we obtain the following result.

\begin{theorem}\label{thm.bootstrap}
Let $L\in \LL_\gamma$ for $\gamma>0$, and let $U$ be the evolution 
system generated by $L$ on $W^{k,p}_{a,w}$.
Let $\cG_t^{[m]}$ is the $m^{th}$-order approximation of the Green 
function for $\pa_t-L(t)$, $m\in \ZZ_+$. Then, if $\omega$ and $M$ 
are the constants in Lemma \ref{lemma.renorm},
\[
  |||U(t,0) - \prod_{k=0}^{n-1}\, \left ( \cG_{(k+1)t/n, kt/n}^{[m]} \right )
  |||_{t,0} \leq M
  \frac{t^{(m+1)/2}}{n^{(m-1)/2}} e^{\omega t}.
\]
\end{theorem}

In particular, we have.

\begin{corollary}
In the hypotheses of Theorem \ref{thm.bootstrap}, if $m\geq 2$, then for $t>0$,
\[
    \lim_{n\to }\prod_{k=1}^n\, \left ( \cG_{(k+1)t/n, kt/n}^{[m]} \right ) \seq U(t,0),
\]
strongly in $W^{k,p}_{a,w}$.
\end{corollary}

\bibliographystyle{plain}

\vspace*{.3in}

\end{document}